\documentclass[12pt]{article}

\usepackage{graphicx}
\usepackage{epstopdf}
\usepackage{hyperref}
\hypersetup{
  colorlinks   = true,    
  urlcolor     = blue,    
  linkcolor    = blue,    
  citecolor    = red      
}
\usepackage[caption=false]{subfig}

\usepackage{amsmath}
\usepackage{csquotes}
\usepackage{tikz}
\usepackage{upgreek}
\usepackage{amssymb}
\usepackage{amsthm}
\usepackage{amsfonts}
\numberwithin{equation}{section}
\usepackage{booktabs,caption}
\usepackage[flushleft]{threeparttable}
\usepackage{stmaryrd}
\usepackage{wasysym}
\usepackage{makecell}
\usepackage{float}
\usepackage{longtable}
\usepackage{enumitem}

\newtheorem{theorem}{Theorem}[section]
\newtheorem{definition}[theorem]{Definition}

\newtheorem{lemma}[theorem]{Lemma}

\newtheorem{corollary}[theorem]{Corollary}
\newtheorem{notation}[theorem]{Notation}
\newtheorem{remark}[theorem]{Remark}

\theoremstyle{definition}
\newtheorem{example}[theorem]{Example}
\newtheorem{conjecture}[theorem]{Conjecture}
\usepackage{nccmath}
\usepackage{mathtools}
\usepackage[reftex]{theoremref}
\usepackage{tikz-cd} 
\usepackage[left=1in, right=1in, top=1in, bottom=1in]{geometry}
\begin{document}


\title{$(\sigma, \tau)$-Derivations of Number Rings with Coding Theory Applications}

\author{Praveen Manju and Rajendra Kumar Sharma}
\date{}
\maketitle

\begin{center}
\noindent{\small Department of Mathematics, \\Indian Institute of Technology Delhi, \\ Hauz Khas, New Delhi-110016, India$^{1}$}
\end{center}

\footnotetext[1]{{\em E-mail addresses:} \url{praveenmanjuiitd@gmail.com}(Corresponding Author: Praveen Manju), \url{rksharmaiitd@gmail.com}(Rajendra Kumar Sharma).}

\medskip

\begin{abstract}
In this article, we study $(\sigma, \tau)$-derivations of number rings by considering them as commutative unital $\mathbb{Z}$-algebras. We begin by characterizing all $(\sigma, \tau)$-derivations and inner $(\sigma, \tau)$-derivations of the ring of algebraic integers of a quadratic number field. Then we characterize all $(\sigma, \tau)$-derivations of the ring of algebraic integers $\mathbb{Z}[\zeta]$ of a $p^{\text{th}}$-cyclotomic number field $\mathbb{Q}(\zeta)$ ($p$ odd rational prime and $\zeta$ a primitive $p^{\text{th}}$-root of unity). We also conjecture (using SageMath and MATLAB) an \enquote{if and only if} condition for a $(\sigma, \tau)$-derivation $D$ on $\mathbb{Z}[\zeta]$ to be inner. We further characterize all $(\sigma, \tau)$-derivations and inner $(\sigma, \tau)$-derivations of the bi-quadratic number ring $\mathbb{Z}[\sqrt{m}, \sqrt{n}]$ ($m$, $n$ distinct square-free rational integers). In each of the above cases, we also determine the rank and an explicit basis of the derivation algebra consisting of all $(\sigma, \tau)$-derivations of the number ring. As a consequence, we solve the twisted derivation problem in the ring of algebraic integers of a quadratic number field and in a bi-quadratic number ring, and we conjecture a solution of the twisted derivation problem in the ring of algebraic integers of a $p^{\text{th}}$-cyclotomic number field. Finally, we give the applications of our work in coding theory by constructing Hom-IDD codes.
\end{abstract}

\textbf{Keywords:}
$(\sigma, \tau)$-derivation, inner $(\sigma, \tau)$-derivation, number ring, ring of algebraic \\ integers, quadratic, cyclotomic, bi-quadratic, coding theory

\textbf{Mathematics Subject Classification (2020):}  Primary: 13N15, 11R04, 94B05; \\ Secondary: 11R11, 11R18, 11R16, 11C20.

\section{Introduction}\label{section 1}
Derivations have been studied in various fields like algebra, functional analysis, Fourier analysis, measure theory, operator theory, harmonic analysis, differential equations, differential geometry, etc. They play an essential role in many important branches of mathematics and physics. The notion of derivation in rings has been applied to and studied in various algebras, yielding their theories of derivations. For example, BCI-algebras \cite{Muhiuddin2012}, von Neumann algebras \cite{Brear1992},  incline algebras which have many applications \cite{Kim2014}, MV-algebras \cite{KamaliArdakani2013, Mustafa2013}, Banach algebras \cite{Raza2016}, lattices that have a very crucial role in various fields like information theory: information recovery, management of information access and cryptanalysis \cite{Chaudhry2011}. Differentiable manifolds, operator algebras, $\mathbb{C}^{*}$-algebras, and the representation theory of Lie groups continue to be studied using derivations \cite{Klimek2021}. Derivations on rings and algebras help in the study of their structure. For a historical account and more applications of derivations, we refer the reader to \cite{Atteya2019, Haetinger2011, MohammadAshraf2006}. The idea of an $(s_{1}, s_{2})$-derivation was introduced by Jacobson \cite{Jacobson1964}. These derivations were later on commonly called as $(\sigma, \tau)$ or $(\theta, \phi)$-derivation. These derivations have been highly studied in prime and semiprime rings and have been principally used in solving functional equations \cite{Brear1992}.  Twisted derivations have huge applications. They are used in multiplicative deformations and discretizations of derivatives that have many applications in models of quantum phenomena and the analysis of complex systems and processes. They are extensively investigated in physics and engineering. Using twisted derivations, Lie algebras are generalized to hom-Lie algebras, and the central extension theory is developed for hom-Lie algebras analogously to that for Lie algebras. Just as Lie algebras were initially studied as algebras of derivations, hom-Lie algebras were defined as algebras of twisted derivations. The generalizations (deformations and analogs) of the Witt algebra, the complex Lie algebra of derivations on the algebra of Laurent polynomials $\mathbb{C}[t, t^{-1}]$ in one variable, are obtained using twisted derivations. Deformed Witt and Virasoro-type algebras have applications in analysis, numerical mathematics, algebraic geometry, arithmetic geometry, number theory, and physics. We refer the reader to \cite{Hartwig2006, ilwale2023noncommutative, Larsson2017, Siebert1996} for details. Twisted derivations have been used to generalize Galois theory over division rings and in the study of $q$-difference operators in number theory. The notion of absolute derivations has also been used in number theory (see \cite{Lagarias2005} and \cite{N.Kurokawa2003}). Derivations, especially $(\sigma, \tau)$-derivations of rings, have various applications in coding theory \cite{Creedon2019, Boucher2014}. For more applications of $(\sigma, \tau)$-derivations, we refer the reader to \cite{AleksandrAlekseev2020}.

The earliest work involving derivations in number fields is seen in A. Weil's paper \cite{Weil1943}, which gave an idea to develop the theory of differents in algebraic number fields on an arithmetical analog of differentiation in function fields. In \cite{Kawada1951} and \cite{Kinohara1952}, the authors use A. Weil's idea and derivations to develop the theory of the relative differents in algebraic number fields. In \cite{Narkiewicz1969}, the author proves a theorem of A. Weil concerning the notion of essential $I$-derivation ($I$ an ideal in the ring of algebraic integers of a finite extension of an algebraic number field). Since then, no literature can be found that explores derivations in number fields, except \cite{Chaudhuri}, where the author studies $(\sigma, \tau)$-derivations of the ring of algebraic integers of a quadratic number field. Twisted derivations in number fields are yet to be explored.
Derivations have been studied, in general, in field extensions as well. The relationship between derivations and field extensions can be found in \cite{Davis1973}. Some more references where derivations in field extensions have been studied are  \cite{Callahan1973, Derksen1993, Deveney1979, Heerema1981, Heerema1962, Messmer1995, Mordeson1972, Suzuki1981, Thwing1974, Ziegler2003}. Also, this article considers the twisted derivation problem: Are all twisted derivations inner? Or is the space of outer twisted derivations trivial? This problem was initially posed for derivations in group algebras. We refer the reader to \cite{AleksandrAlekseev2020, A.A.Arutyunov2020, Arutyunov2020b, Arutyunov2020} for the history and importance of the derivation problem. This article considers the analogous problem for $(\sigma, \tau)$-derivations in the ring of algebraic integers of quadratic, cyclotomic, and bi-quadratic number fields.

Let $R$ be a commutative unital ring and $\mathcal{A}$ be an associative $R$-algebra. Let $(\sigma, \tau)$ be a pair of $R$-algebra endomorphisms of $\mathcal{A}$. A $(\sigma, \tau)$-derivation $D: \mathcal{A} \rightarrow \mathcal{A}$ satisfies the identity \begin{equation*} D(\alpha^{k}) = \left(\sum_{i+j=k-1} \sigma(\alpha^{i}) \tau(\alpha^{j})\right) D(\alpha)\end{equation*} for all $\alpha \in \mathcal{A}$ and for all $k \in \mathbb{N}$, where $i,j$ run over non-negative integers. A natural question one may ask is: under what assumptions on an $R$-linear map $D:\mathcal{A} \rightarrow \mathcal{A}$ satisfying the above relations is a $(\sigma, \tau)$-derivation? Similar questions have been posed for derivations of a ring (for example, see \cite{Bridges1984, hosseini2018identities, Vukman2005}). We study this question in number rings. In this article, we mainly study $(\sigma, \tau)$-derivations of number rings. We partition the article into five sections. In Section \ref{subsection 2.3}, we first obtain results on $(\sigma, \tau)$-derivations and inner $(\sigma, \tau)$-derivations of a commutative unital algebra $\mathcal{A}$ over a commutative ring $R$ with unity $1$. We find a necessary condition for an $R$-linear map $D$ on $\mathcal{A}$ to be a $(\sigma, \tau)$-derivation and give counterexamples showing that the condition is not sufficient. This motivates us to study these derivations in certain number rings where this condition is both necessary as well as sufficient. We also determine a necessary and sufficient condition for a $(\sigma, \tau)$-derivation on $\mathcal{A}$ to be inner. In Section \ref{section 3}, we apply our results obtained in Section \ref{subsection 2.3} to study $(\sigma, \tau)$-derivations and inner $(\sigma, \tau)$-derivations of number rings by considering them as commutative unital $\mathbb{Z}$-algebras. Section \ref{section 3} is further subdivided into three subsections. In Subsection \ref{subsection 3.1}, we study $(\sigma, \tau)$-derivations and inner $(\sigma, \tau)$-derivations of the ring of algebraic integers of a quadratic number field. We obtain necessary and sufficient conditions for a $\mathbb{Z}$-linear map $D$ on the ring of algebraic integers of a quadratic field to be a $(\sigma, \tau)$-derivation. We also classify the inner $(\sigma, \tau)$-derivations of the ring of algebraic integers of a quadratic field. In Subsection \ref{subsection 3.2}, we study $(\sigma, \tau)$-derivations and inner $(\sigma, \tau)$-derivations of the ring of algebraic integers of a cyclotomic number field. We obtain necessary and sufficient conditions for a $\mathbb{Z}$-linear map $D$ on the ring of algebraic integers $\mathbb{Z}[\zeta]$ of a  $p^{\text{th}}$ cyclotomic field $\mathbb{Q}(\zeta)$ to be a $(\sigma, \tau)$-derivation. We prove that $\mathcal{D}_{(\sigma, \tau)}(\mathbb{Z}[\zeta])$ is a $\mathbb{Z}$-module having rank $p-1$. Further, in this process, we also propose two conjectures giving a necessary and sufficient condition on a $(\sigma, \tau)$-derivation of $\mathbb{Z}[\zeta]$ to be inner. This is done with the help of SageMath and MATLAB. In Subsection \ref{subsection 3.3}, a beautiful application of \th\ref{lemma 2.1} is given in determining all $(\sigma, \tau)$-derivations of the number ring $\mathbb{Z}[\sqrt{m}, \sqrt{n}]$ for every pair $(\sigma, \tau)$ of distinct ring endomorphisms $\sigma$ and $\tau$ of $\mathbb{Z}[\sqrt{m}, \sqrt{n}]$. As a result, we establish that $D_{(\sigma, \tau)}(\mathbb{Z}[\sqrt{m}, \sqrt{n}])$ is a $\mathbb{Z}$-module having rank $4$. Further, we obtain a necessary and sufficient condition for a $(\sigma, \tau)$-derivation of $\mathbb{Z}[\sqrt{m}, \sqrt{n}]$ to be inner. Finally, in Section \ref{section 4}, we discuss the coding theory applications of our work. We introduce the notion of Hom-IDD code. In Section \ref{section 5}, we conclude our findings.

\section{Preliminaries}
\subsection{Some Basic Definitions and Results on $(\sigma, \tau)$-Derivations}\label{subsection 2.1}
Let $R$ be a commutative unital ring and $\mathcal{A}$ be an associative $R$-algebra. Let $(\sigma, \tau)$ be a pair of $R$-algebra endomorphisms of $\mathcal{A}$.

\begin{definition}\label{definition 2.1}
An $R$-linear map $d:\mathcal{A} \rightarrow \mathcal{A}$ that satisfies $d(\alpha \beta) = d(\alpha) \beta + \alpha d(\beta)$ for all $\alpha, \beta \in \mathcal{A}$, is called a derivation of $\mathcal{A}$. It is called inner if there exists some $\beta \in \mathcal{A}$ such that $d(\alpha) = \beta \alpha - \alpha \beta$ for all $\alpha \in \mathcal{A}$, and then we denote it by $d_{\beta}$. The elements of the quotient of the $R$-module of all derivations of $\mathcal{A}$ by the $R$-submodule of all inner derivations are called outer derivations.
\end{definition}

\begin{definition}\label{definition 2.2}
A $(\sigma, \tau)$-derivation $D:\mathcal{A} \rightarrow \mathcal{A}$ is an $R$-linear map that satisfies the $(\sigma, \tau)$-twisted generalized identity: $D(\alpha \beta) = D(\alpha)\tau(\beta) + \sigma(\alpha) D(\beta)$ for all $\alpha, \beta \in \mathcal{A}$. It is called inner if there exists some $\beta \in \mathcal{A}$ such that $D(\alpha) = \beta \tau(\alpha) - \sigma(\alpha) \beta$ for all $\alpha \in \mathcal{A}$, and then we denote it by $D_{\beta}$. The elements of the quotient of the $R$-module of all $(\sigma, \tau)$-derivations of $\mathcal{A}$ by the $R$-submodule of all inner $(\sigma, \tau)$-derivations are called outer $(\sigma, \tau)$-derivations.
\end{definition}

\begin{notation}\label{notation 2.3}
We denote the set of all $(\sigma, \tau)$-derivations on $\mathcal{A}$ by $\mathcal{D}_{(\sigma, \tau)}(\mathcal{A})$. We denote the corresponding set of inner $(\sigma, \tau)$-derivations on $\mathcal{A}$ by $\text{Inn}_{(\sigma, \tau)}(\mathcal{A})$. We denote the corresponding set of outer $(\sigma, \tau)$-derivations on $\mathcal{A}$ by $\text{Out}_{(\sigma, \tau)}(\mathcal{A})$.
\end{notation}

\begin{remark}\label{remark 2.4}
If $\sigma = \tau = id_{\mathcal{A}}$ (the identity map on $\mathcal{A}$), then the usual Leibniz identity holds, and $(\sigma, \tau)$-derivation, inner $(\sigma, \tau)$-derivation, and outer $(\sigma, \tau)$-derivation respectively become the ordinary derivation, ordinary inner derivation and ordinary outer derivation on $\mathcal{A}$. Defining componentwise sum and module action, $\mathcal{D}_{(\sigma, \tau)}(\mathcal{A})$ becomes an $R$- as well as $\mathcal{A}$-module, and  $\text{Inn}_{(\sigma, \tau)}(\mathcal{A})$ become its submodule. If $1$ is the unity in $\mathcal{A}$ and $D$ is a $(\sigma, \tau)$-derivation of $\mathcal{A}$, then $D(1) = 0$.
\end{remark}

In view of the the above Definition \ref{definition 2.2} and the Remark \ref{remark 2.4}, the outer $(\sigma, \tau)$-derivations are precisely the elements of the factor module $\text{Out}_{(\sigma, \tau)}(\mathcal{A}) = \frac{\mathcal{D}_{(\sigma, \tau)}(\mathcal{A})}{\text{Inn}_{(\sigma, \tau)}(\mathcal{A})}$. Also, note that the set $\mathcal{D}_{(\sigma, \tau)}(\mathcal{A}) \setminus \text{Inn}_{(\sigma, \tau)}(\mathcal{A})$ is the set of all non-inner $(\sigma, \tau)$-derivations of $\mathcal{A}$. The following lemma establishes a connection between our notions of outer $(\sigma, \tau)$-derivations and non-inner $(\sigma, \tau)$-derivations of $\mathcal{A}$. Its proof can be found in \cite{Manju2024}.
\begin{lemma}\label{lemma 2.5}
Let $T = \{D_{i} \in \mathcal{D}_{(\sigma, \tau)}(\mathcal{A}) \mid i \in I\}$ ($I$ some indexing set) be a left transversal of $\text{Inn}_{(\sigma, \tau)}(\mathcal{A})$ in $\mathcal{D}_{(\sigma, \tau)}(\mathcal{A})$ with $0$ as the coset representative of the coset $\text{Inn}_{(\sigma, \tau)}(\mathcal{A})$. Then the non-inner $(\sigma, \tau)$-derivations of $\mathcal{A}$ correspond to the elements in the set $\bigcup_{D_{i} \in T \setminus \{0\}} (D_{i} + \text{Inn}_{(\sigma, \tau)}(\mathcal{A}))$. More precisely, $\mathcal{D}_{(\sigma, \tau)}(\mathcal{A}) \setminus \text{Inn}_{(\sigma, \tau)}(\mathcal{A}) = \bigcup_{D_{i} \in T \setminus \{0\}} (D_{i} + \text{Inn}_{(\sigma, \tau)}(\mathcal{A}))$.
\end{lemma}

\begin{remark}\label{remark 2.6}
In view of Lemma \ref{lemma 2.5}, $D$ is a non-inner derivation of $\mathcal{A}$ if and only if $D + \text{Inn}_{(\sigma, \tau)}(\mathcal{A})$ is a non-zero element of $\text{Out}_{(\sigma, \tau)}(\mathcal{A}) = \frac{\mathcal{D}_{(\sigma, \tau)}(\mathcal{A})}{\text{Inn}_{(\sigma, \tau)}(\mathcal{A})}$. In other words, $D$ is a non-inner derivation of $\mathcal{A}$ if and only if $D + \text{Inn}_{(\sigma, \tau)}(\mathcal{A})$ is a non-trivial outer derivation of $\mathcal{A}$. Therefore, studying the non-trivial outer derivations of $\mathcal{A}$ is equivalent to studying the non-inner derivations of $\mathcal{A}$ (see {\cite[Chapter 11]{pierce}} for details).
\end{remark} 

Many authors have defined an outer derivation of a ring $R$ (or algebra $\mathcal{A}$) to be a non-inner derivation of $R$ (or $\mathcal{A}$). Some references in this regard are \cite{batty1978derivations, chuang2005identities, dhara2022note, eroǧlu2017images, hall1972derivations, miles1964derivations, prajapati2022b, sakai1966derivations, weisfeld1960derivations}. In \cite{arutyunov2019smooth} and \cite{mishchenko2020description}, the authors initially describe the set $\mathcal{D}(\mathcal{A}) \setminus \text{Inn}(\mathcal{A})$ as the set of outer derivations but then argue that it is more natural to call the quotient module $\text{Out}(\mathcal{A}) = \frac{\mathcal{D}(\mathcal{A})}{\text{Inn}(\mathcal{A})}$ as the set of outer derivations because this module can be interpreted as $1^{\text{st}}$ Hochschild cohomology module of $\mathcal{A}$ with coefficients in $\mathcal{A}$. Similarly, the quotient module $\text{Out}_{(\sigma, \tau)}(\mathcal{A}) = \frac{\mathcal{D}_{(\sigma, \tau)}(\mathcal{A})}{\text{Inn}_{(\sigma, \tau)}(\mathcal{A})}$ is our set of outer $(\sigma, \tau)$-derivations because this module can be interpreted as $1^{\text{st}}$ $(\sigma, \tau)$-Hochschild cohomology module of $\mathcal{A}$ with coefficients in $\mathcal{A}$ (see \cite{ilwale2023noncommutative}).

\begin{definition}\label{definition 2.7}
$\mathcal{A}$ is said to be $(\sigma, \tau)$-differentially trivial if $\mathcal{A}$ has only zero $(\sigma, \tau)$-derivation, that is, $\mathcal{D}_{(\sigma, \tau)}(\mathcal{A}) = \{0\}$.
\end{definition}

\begin{definition}\label{definition 2.8}
If $R$ is a ring and $G$ is a group, then the group ring of $G$ over $R$ is defined as the set $$RG = \{\sum_{g \in G} a_{g} g \mid a_{g} \in R, \forall g \in G \hspace{0.2cm} \text{and} \hspace{0.2cm} |\text{supp}(\alpha)| < \infty \},$$ where for $\alpha = \sum_{g \in G} a_{g} g$, $\text{supp}(\alpha)$ denotes the support of $\alpha$ that consists of elements from $G$ that appear in the expression of $\alpha$. The set $RG$ is a ring concerning the componentwise addition and multiplication defined respectively by: For $\alpha = \sum_{g \in G} a_{g} g$, $\beta = \sum_{g \in G} b_{g} g$ in $RG$, $$(\sum_{g \in G} a_{g} g ) + (\sum_{g \in G} b_{g} g) = \sum_{g \in G}(a_{g} + b_{g}) g \hspace{0.2cm} \text{and} \hspace{0.2cm} \alpha \beta = \sum_{g, h \in G} a_{g} b_{h} gh.$$ If the ring $R$ is commutative having unity $1$ and the group $G$ is abelian having identity $e$, then $RG$ becomes a commutative unital algebra over $R$ with identity $1 = 1e$. 
\end{definition}

If $\sigma$ and $\tau$ are algebra endomorphisms of $RG$ which are $R$-linear extensions of group endomorphisms of $G$ and $D:RG \rightarrow RG$ is a $(\sigma, \tau)$-derivation, then $$D(e) = D(ee) = D(e) \tau(e) + D(e) \sigma(e) = D(e) + D(e)$$ so that $D(e) = 0$.

\subsection{Some Basic Definitions and Results from Algebraic Number Theory}
We shall need some definitions and results from algebraic number theory, for which we refer the reader to \cite{Marcus2018, IanStewart2002}. Still, we summarize a few, as follows: A complex number is an algebraic number if it is a zero of some non-constant polynomial over $\mathbb{Q}$. A number field is a finite field extension of $\mathbb{Q}$ that contains algebraic numbers. A number ring is any subring of a number field. An algebraic integer is a complex number that is a zero of some monic polynomial over $\mathbb{Z}$. The ring of algebraic integers of a number field $K$, denoted by $O_{K}$, is defined as $O_{K} = K \cap \mathbb{B}$, where $\mathbb{B}$ is the ring of all algebraic integers. Thus $O_{K}$ contains precisely all those algebraic integers which belong to $K$. A subring of $O_{K}$ which is also a $\mathbb{Z}$-module of rank $[K: \mathbb{Q}]$ is called an order of $K$. A set of algebraic numbers $\{\alpha_{1}, \alpha_{2}, ..., \alpha_{r}\}$ is called an integral basis of $K$ (or $O_{K}$) if each element in $O_{K}$ is uniquely expressible as a $\mathbb{Z}$-linear combination of $\alpha_{1}, \alpha_{2}, ..., \alpha_{r}$. Every integral basis is a $\mathbb{Q}$-basis. An integral basis having form $\{1, \theta, ..., \theta^{n-1}\}$ for some $\theta \in O_{K}$, is called a power basis of $K$. A number field $K$ that has a power basis is called monogenic. A number field $K$ that has degree $2$ over $\mathbb{Q}$ is called quadratic. A number field that has the form $K = \mathbb{Q}(\zeta)$, where $\zeta = e^{2 \pi i / m}$ ($m \geq 1$) is a primitive complex $m^{\text{th}}$ root of unity, is called an $m^{\text{th}}$ cyclotomic field. We state below some results that we may need later.

\begin{theorem}[{\cite[Theorem 2.2]{IanStewart2002}}]\th\label{theorem 1.1}
Let $K$ be a number field. Then $K = \mathbb{Q}(\theta)$ for some algebraic number $\theta$, or in other words, number fields are finite simple extensions of $\mathbb{Q}$. In fact, $K = \mathbb{Q}(\theta)$ for some algebraic integer $\theta$.
\end{theorem}

\begin{theorem}[{\cite[Theorem 2.16]{IanStewart2002}}]\th\label{theorem 1.2}
A number field $K$ always has an integral basis. Further, the additive group of $O_{K}$ is free abelian, having rank equal to the degree of $K$.
\end{theorem}

\begin{theorem}[{\cite[Proposition 3.1]{IanStewart2002}}]\th\label{theorem 1.3}
Every quadratic field is of the form $\mathbb{Q}(\sqrt{d})$ for some square-free rational integer $d$.
\end{theorem}

\begin{theorem}[{\cite[Theorem 3.2]{IanStewart2002}}]\th\label{theorem 1.4}
Let $K = \mathbb{Q}(\sqrt{d})$ be a quadratic field. Then
\begin{itemize}
\item[(i)] $O_{K} = \mathbb{Z}[\sqrt{d}]$ when $d \not\equiv 1 \hspace{0.1cm} (\text{mod} \hspace{0.1cm} 4)$,
\item[(ii)] $O_{K} = \mathbb{Z}[\frac{1 + \sqrt{d}}{2}]$ when $d \equiv 1 \hspace{0.1cm} (\text{mod} \hspace{0.1cm} 4)$.
\end{itemize}
\end{theorem}

\begin{theorem}[{\cite[Theorem 3.3]{IanStewart2002}}]\th\label{theorem 1.5}
Let $K = \mathbb{Q}(\sqrt{d})$ be a quadratic field. Then $\mathcal{B}$ is an integral basis of $K$, where
\begin{itemize}
\item[(a)] $\mathcal{B} = \{1, \sqrt{d}\}$ when $d \not\equiv 1 \hspace{0.1cm} (\text{mod} \hspace{0.1cm} 4)$.
\item[(b)] $\mathcal{B} = \{1, \frac{1 + \sqrt{d}}{2}\}$ when $d \equiv 1 \hspace{0.1cm} (\text{mod} \hspace{0.1cm} 4)$.
\end{itemize}
\end{theorem}

\begin{theorem}[{\cite[Lemma 3.4]{IanStewart2002}}]\th\label{theorem 1.6}
If $p$ is an odd prime and $\zeta$ is a primitive $p^{\text{th}}$ root of unity, then the minimal polynomial of $\zeta$ over $\mathbb{Q}$ is $\phi_{p}(x) = x^{p-1} + x^{p-2} + ... + 1$. Hence $[\mathbb{Q}(\zeta):Q] = p-1$.
\end{theorem}

\begin{theorem}[{\cite[Theorem 3.5]{IanStewart2002}}]\th\label{theorem 1.7}
Let $p$ be an odd prime and $K = \mathbb{Q}(\zeta)$ be a $p^{\text{th}}$ cyclotomic field. Then $O_{K} = \mathbb{Z}[\zeta]$. Hence, $\{1, \zeta, \zeta^{2}, ..., \zeta^{p-2}\}$ is an integral basis of $K = \mathbb{Q}(\zeta)$.
\end{theorem}

\begin{theorem}[{\cite[Chapter 2]{Marcus2018}}]\th\label{theorem 1.8}
Let $m \in \mathbb{N}$ and $K = \mathbb{Q}(\zeta)$ be an $m^{\text{th}}$ cyclotomic field, then $O_{K} = \mathbb{Z}[\zeta]$.
\end{theorem}

\subsection{Some Useful Results on \texorpdfstring{$(\sigma, \tau)$}{Lg}-Derivations of Commutative \\ Algebras}\label{subsection 2.3}
$\mathcal{A}$, in this section, denotes a commutative algebra with unity $1$ over a commutative ring $R$ having unity $1$ and $\sigma$, $\tau$ are two different non-zero unital $R$-algebra endomorphisms of $\mathcal{A}$. First, we have an important lemma below.

\begin{lemma}\th\label{lemma 2.1}
Let $\mathcal{A}$ be of finite rank $n$ as an $R$-module and let $\{\alpha_{1}, \alpha_{2}, ..., \alpha_{n}\}$ be an $R$-basis of $\mathcal{A}$. Then an $R$-linear map $D:\mathcal{A} \rightarrow \mathcal{A}$ is a $(\sigma, \tau)$-derivation if and only if $$D(\alpha_{i} \alpha_{j}) = D(\alpha_{i}) \tau(\alpha_{j}) + \sigma(\alpha_{i}) D(\alpha_{j})$$ for every $i, j \in \{1, 2, ..., n\}$.
\end{lemma}

\begin{proof} Let $D:\mathcal{A} \rightarrow \mathcal{A}$ be an $R$-linear map, and $x, y \in \mathcal{A}$. Then $x = \sum_{i=1}^{n} a_{i} \alpha_{i}$ and $y = \sum_{j=1}^{n} b_{j} \alpha_{j}$ for some $a_{i}, b_{j} \in R$ ($i, j \in \{1, 2, ..., n\}$). Observe that $$xy = \left(\sum_{i=1}^{n} a_{i} \alpha_{i}\right) \left(\sum_{j=1}^{n} b_{j} \alpha_{j} \right) = \sum_{i, j} (a_{i}b_{j}) (\alpha_{i} \alpha_{j}),$$ and

\begin{eqnarray*}
D(xy) & = & \sum_{i,j=1}^{n} (a_{i}b_{j}) D(\alpha_{i} \alpha_{j}) = \sum_{i,j=1}^{n} (a_{i}b_{j}) (D(\alpha_{i}) \tau(\alpha_{j}) + \sigma(\alpha_{i}) D(\alpha_{j})) \\ & = & D \left(\sum_{i=1}^{n} a_{i}\alpha_{i}\right) \tau \left(\sum_{j=1}^{n} b_{j} \alpha_{j}\right) + \sigma \left(\sum_{i=1}^{n} a_{i}\alpha_{i}\right) D \left(\sum_{j=1}^{n} b_{j} \alpha_{j}\right) \\ & = & D(x) \tau(y) + \sigma(x) D(y). \end{eqnarray*}

Hence $D$ is a $(\sigma, \tau)$-derivation. The converse is straightforward.
\end{proof}

For a non-negative integer $k$, define $S_{k}$ as the set containing all ordered pairs $(i,j)$ ($i, j \in \mathbb{N} \cup \{0\}$) for which $i+j = k$. Note that the sets $S_{k} \text{'s}$ are pairwise disjoint. Therefore, $$|\cup_{i=0}^{m} S_{i}| = \sum_{i=0}^{m} |S_{i}| = 1 + 2 + 3 + ... + m + (m+1) = \frac{(m+1)(m+2)}{2}$$ for every non-negative integer $m$.

\begin{lemma}\th\label{lemma 2.2}
For a  $(\sigma, \tau)$-derivation $D$ of $\mathcal{A}$, $$D(\alpha^{k}) = \left(\sum_{(i,j) \in S_{k-1}}  \sigma(\alpha^{i}) \tau(\alpha^{j})\right)D(\alpha),$$ for every $\alpha \in \mathcal{A}$ and for every $k \in \mathbb{N}$.
\end{lemma}

\begin{proof} We use induction. For $k = 1$, the equality trivially holds. Let the result hold for $k=n$, that is, $D(\alpha^{n}) = \left(\sum_{(i,j) \in S_{n-1}}  \sigma(\alpha^{i}) \tau(\alpha^{j})\right)D(\alpha)$. Then

\begin{eqnarray*}
D(\alpha^{n+1}) = D(\alpha^{n} \alpha) & = & D(\alpha^{n}) \tau(\alpha) + \sigma(\alpha^{n}) D(\alpha) \\ & = & \left(\sum_{i+j = n-1} \sigma(\alpha^{i}) \tau(\alpha^{j+1}) + \sigma(\alpha^{n}) \tau(\alpha^{0})\right) D(\alpha) \\ & = & \left(\sum_{i+j = n} \sigma(\alpha^{i}) \tau(\alpha^{j})\right)D(\alpha) \\ & = & \left(\sum_{(i,j) \in S_{n}}  \sigma(\alpha^{i}) \tau(\alpha^{j})\right)D(\alpha).
\end{eqnarray*}

Induction is complete, and so is proof.
\end{proof}

The theorem below is now straightforward.

\begin{theorem}\th\label{theorem 2.3}
Let $\mathcal{A}$ be of finite rank $n$ and suppose that $\mathcal{A}$ has an $R$-basis of the form $\{1, \alpha, \alpha^{2}, ..., \alpha^{n-1}\}$ for some $\alpha \in \mathcal{A}$. If an $R$-linear map $D:\mathcal{A} \rightarrow \mathcal{A}$ is a $(\sigma, \tau)$-derivation, then \begin{equation}D(\alpha^{k}) = \left(\sum_{(i,j) \in S_{k-1}}  \sigma(\alpha^{i}) \tau(\alpha^{j})\right)D(\alpha)\end{equation} for all $k \in \{1, 2, ..., n-1\}$.
\end{theorem}

The converse of the \th\ref{theorem 2.3} is not true. Some examples are given below.\vspace{10pt}

\begin{example}
Let $R[X]$ be the polynomial ring in variable $X$ over a commutative unital ring $R$. Let $f(X) \in R[X]$ be monic with degree $n$ and $\mathcal{A} = \frac{R[X]}{\langle f(X) \rangle}$. Then $\mathcal{A}$ is a commutative $R$-algebra with unity $\bar{1} = 1 + f(X)$. Suppose $\alpha = X + f(X)$. Then $\{1, \alpha, \alpha^{2}, ..., \alpha^{n-1}\}$ is an $R$-basis of the $R$-module $\mathcal{A}$. We denote the zero element $0 + \langle f(X) \rangle$ of $\mathcal{A}$ by $\bar{0}$.

In particular, take $R = \mathbb{Z}$ and $f(X) = X^{4} - 1$ so that $\{\bar{1}, \alpha, \alpha^{2}, \alpha^{3}\}$ is a $\mathbb{Z}$-basis of $\mathcal{A} = \frac{R[X]}{\langle f(X) \rangle}$. Let $\sigma$ and $\tau$ be $\mathbb{Z}$-algebra endomorphisms of $\mathcal{A}$ given by $\sigma(\alpha) = \alpha$ and $\tau(\alpha) = \alpha^{2}$. Then $\sigma(\bar{1}) = \bar{1}$ and $\tau(\bar{1}) = \bar{1}$. Define $D:\mathcal{A} \rightarrow \mathcal{A}$ as a $\mathbb{Z}$-linear map with $D(\bar{1}) = \bar{0}$ and $$D(\alpha^{r}) = \left(\sum_{(i,j) \in S_{r-1}}  \sigma(\alpha^{i}) \tau(\alpha^{j})\right)D(\alpha), \hspace{0.1cm} \forall \hspace{0.1cm} r \in \{1, 2, 3\}.$$

$\alpha^{4} = \bar{1}$ so that $D(\alpha^{4}) = \bar{0}$. Also, 
\begin{eqnarray*} \left(\sum_{(i,j) \in S_{3}}  \sigma(\alpha^{i}) \tau(\alpha^{j})\right)D(\alpha) & = & \left(\sigma(\alpha^{3}) + \sigma(\alpha^{2}) \tau(\alpha) + \sigma(\alpha) \tau(\alpha^{2}) + \tau(\alpha^{3})\right)D(\alpha) \\ & = & \left(\alpha^{3} + \alpha^{4} + \alpha^{5}  + \alpha^{6}\right)D(\alpha) \\ & = & \left(\bar{1} + \alpha + \alpha^{2} + \alpha^{3}\right)D(\alpha).
\end{eqnarray*}
Now $\bar{1} + \alpha + \alpha^{2} + \alpha^{3} \neq \bar{0}$ since $\{\bar{1}, \alpha, \alpha^{2}, \alpha^{3}\}$ is a basis of $\mathcal{A}$. Also, $\mathcal{A}$ is an integral domain, so if $D(\alpha) \neq 0$, then the above expression cannot be zero. Further, observe that $$D(\alpha^{4}) \neq \left(\sum_{(i,j) \in S_{3}}  \sigma(\alpha^{i}) \tau(\alpha^{j})\right)D(\alpha).$$ Hence by \th\ref{lemma 2.2}, $D$ cannot be a $(\sigma, \tau)$-derivation of $\mathcal{A}$. 
\end{example}

\begin{example}
Consider the set $$\mathcal{A} = \{aI + bA + cA^{2} \mid a, b, c \in \mathbb{Z}\},$$ where $A = \begin{pmatrix}
0 & 0 & 1 \\
1 & 0 & 0 \\
0 & 1 & 0
\end{pmatrix}$. Then $\mathcal{A}$ is a commutative subalgebra of the $\mathbb{Z}$-algebra $M_{3}(\mathbb{Z})$ with unity $I$ and $\{I, A, A^{2}\}$ is a $\mathbb{Z}$-basis of $\mathcal{A}$. Let $\sigma$ and $\tau$ be $\mathbb{Z}$-algebra endomorphisms of $\mathcal{A}$ given by $\sigma(aI + bA + cA^{2}) = aI + bA + cA^{2}$ and $\tau(aI + bA + cA^{2}) = aI + bA^{2} + cA$ for all $a, b, c \in \mathbb{Z}$. Then $\sigma(I) = \tau(I) = I$. Define $D:\mathcal{A} \rightarrow \mathcal{A}$ as a $\mathbb{Z}$-linear map with $D(I) = 0$ and such that $$D(A^{r}) = \left(\sum_{(i,j) \in S_{r-1}} \sigma(A^{i}) \tau(A^{j})\right)D(A), \hspace{0.8cm} \text{for} \hspace{0.2cm} r = 1, 2.$$

Then $D(A^{3}) = D(I) = 0$. Also, 
\begin{eqnarray*}\left(\sum_{(i,j) \in S_{2}}  \sigma(A^{i}) \tau(A^{j})\right)D(A) & = & \left(\sigma(A^{2}) + \sigma(A) \tau(A) + \tau(A^{2})\right)D(A) = \left(I + A + A^{2}\right)D(A).\end{eqnarray*}
Again, $I + A + A^{2} \neq 0$, so if $D(A) \in \mathcal{A}$ is such that $\left(I + A + A^{2}\right)D(A) \neq 0$, for example, $D(A) = A$ or $A^{2}$, then $D(A^{3}) \neq \left(\sum_{(i,j) \in S_{2}}  \sigma(A^{i}) \tau(A^{j})\right)D(A)$. Hence, $D$ is not a $(\sigma, \tau)$-derivation. \end{example}

\begin{example}
Let $n \in \mathbb{N}$ and $A \in M_{n}(\mathbb{Z})$ be an idempotent matrix. Then the subset $$\mathcal{A} = \{aI + bA \mid a, b \in \mathbb{Z}\}$$ of $M_{n}(\mathbb{Z})$ is a commutative $\mathbb{Z}$-algebra with unity $I$ and $\{I, A\}$ is a basis of $\mathcal{A}$. Let $\sigma$ and $\tau$ be $\mathbb{Z}$-algebra endomorphisms of $\mathcal{A}$ given by $\sigma(aI + bA) = (a+b)I$ and $\tau(aI + bA) = aI + bA$ for all $a, b \in \mathbb{Z}$. Define $D:\mathcal{A} \rightarrow \mathcal{A}$ as a $\mathbb{Z}$-linear map with $D(I) = 0$ and such that $$D(A^{r}) = \left(\sum_{(i,j) \in S_{r-1}}  \sigma(A^{i}) \tau(A^{j})\right)D(A), \hspace{0.8cm} \text{for} \hspace{0.2cm} r = 1.$$ $$\left(\sum_{(i,j) \in S_{1}}  \sigma(A^{i}) \tau(A^{j})\right)D(A) = \left(\sigma(A) + \tau(A)\right)D(A) =  \left(I + A\right)D(A).$$
$D(A) \neq \left(I + A\right)D(A)$ provided $A D(A) \neq 0$. Hence, $D$ is not a $(\sigma, \tau)$-derivation of $\mathcal{A}$ provided $D(A)$ is chosen in $M_{n}(\mathbb{Z})$ such that $A D(A) \neq 0$.
\end{example}

\begin{example}
Let $n \in \mathbb{N}$ and $A \in M_{n}(\mathbb{Z})$ be a nilpotent matrix, that is, $A^{k} = 0$ for some least positive integer $k$. Then the subset $$\mathcal{A} = \{\sum_{i=0}^{k-1} a_{i} A^{i} \mid a_{i} \in \mathbb{Z}, \hspace{0.1cm} \forall \hspace{0.1cm} i \in \{0, 1, ..., k-1\}\}$$ of $M_{n}(\mathbb{Z})$ is a commutative $\mathbb{Z}$-algebra with unity $I$ and $\{A^{0}=I, A, ..., A^{k-1}\}$ forms a basis of the $\mathbb{Z}$-module $\mathcal{A}$. Let $\sigma$ and $\tau$ be $\mathbb{Z}$-algebra endomorphisms of $\mathcal{A}$ given by $\sigma(\sum_{i=0}^{k-1} a_{i} A^{i}) = \sum_{i=0}^{k-1} a_{i} A^{i}$ and $\tau(\sum_{i=0}^{k-1} a_{i} A^{i}) = \sum_{i=0}^{k-1} a_{i} B^{i}$ for all $a_{i} \in \mathbb{Z}$ $(0 \leq i \leq k-1),$ where $B = m A$ for some $m \in \mathbb{N}$ and $m \neq 1$. Define $D:\mathcal{A} \rightarrow \mathcal{A}$ as a $\mathbb{Z}$-linear map with $D(I) = 0$ and such that $$D(A^{r}) = \left(\sum_{(i,j) \in S_{r-1}}  \sigma(A^{i}) \tau(A^{j})\right)D(A), \hspace{0.8cm} \text{for} \hspace{0.2cm} r \in \{1, ..., k-1\}.$$

Then $D(A^{k}) = 0$. Also, \begin{eqnarray*}\left(\sum_{(i,j) \in S_{k-1}} \sigma(A^{i}) \tau(A^{j})\right)D(A) = \left(\sum_{i+j=k-1} A^{i} B^{j}\right)D(A) & = & \left(\sum_{j=0}^{k-1} m^{j} A^{k-1} \right)D(A) \\ & = & \left(\frac{1-m^{k}}{1-m}\right)A^{k-1}D(A).\end{eqnarray*}
$0 \neq \left(\frac{1-m^{k}}{1-m}\right)A^{k-1}D(A)$ provided $A^{k-1} D(A) \neq 0$. Hence, $D$ cannot be a $(\sigma, \tau)$-derivation of $\mathcal{A}$ provided $D(A)$ is chosen in $M_{n}(\mathbb{Z})$ such that $A^{k-1} D(A) \neq 0$.
\end{example}

\begin{example}
Let $C_{n} = \langle g \mid g^{n} = 1\rangle$ be a cyclic group having order $n \geqslant 2$ and consider the group ring $\mathbb{Z}C_{n}$. Let $\sigma$ and $\tau$ be $\mathbb{Z}$-algebra endomorphisms of $\mathbb{Z}G$ which are $\mathbb{Z}$-linear extensions of group homomorphisms of $G$, say, $\sigma(g) = 1$ and $\tau(g) = g$. Let $D:\mathbb{Z}C_{n} \rightarrow \mathbb{Z}C_{n}$ be a $\mathbb{Z}$-linear map with $D(1) = 0$ and $$D(g^{r}) = \left(\sum_{(i,j) \in S_{r-1}}  \sigma(g^{i}) \tau(g^{j})\right)D(g), \hspace{0.1cm} \forall \hspace{0.1cm} r \in \{1, ..., n-1\}.$$

So $D(g^{n}) = D(1) = 0$. Also,
\begin{equation*}
\begin{aligned}
\left(\sum_{(i,j) \in S_{n-1}}  \sigma(g^{i}) \tau(g^{j})\right)D(g) & = (\sigma(g^{n-1}) + \sigma(g^{n-2})\tau(g) + \sigma(g^{n-3})\tau(g^{2}) + ... \\ &\quad + \sigma(g^{2}) \tau(g^{n-3}) + \sigma(g) \tau(g^{n-2}) + \tau(g^{n-1}))D(g) 
\\ & = \left(1 + g + ... + g^{n-1}\right)D(g)
\end{aligned}
\end{equation*}
Note that $1 + g + ... + g^{n-1} \neq 0$ since $1, g, ..., g^{n-1}$ are linearly independent being basis elements of $\mathbb{Z}C_{n}$. So if $D(g) \in \mathbb{Z}C_{n}$ is such that $\left(1 + g + ... + g^{n-1}\right)D(g) \neq 0$, then $D(g^{n}) \neq \left(\sum_{(i,j) \in S_{n-1}}  \sigma(g^{i}) \tau(g^{j})\right)D(g)$. Hence, $D$ cannot be a $(\sigma, \tau)$-derivation of $\mathbb{Z}C_{n}$.
\end{example}

\begin{lemma}\th\label{lemma 2.4}
Let $\mathcal{A}$ be of finite rank $n$ and $\{\alpha_{1}, \alpha_{2}, ..., \alpha_{n}\}$ be an $R$-basis of $\mathcal{A}$. Let $D:\mathcal{A} \rightarrow \mathcal{A}$ be a $(\sigma, \tau)$-derivation. Suppose that there exists some $\beta \in \mathcal{A}$ such that $D(\alpha_{i}) = \beta (\tau - \sigma)(\alpha_{i})$ for all $i \in \{1, 2, ..., n\}$. Then $D(\alpha) = \beta (\tau - \sigma)(\alpha)$ for all $\alpha \in \mathcal{A}$.
\end{lemma}

\begin{proof} For $\alpha = \sum_{i=1}^{n} a_{i} \alpha_{i} \in \mathcal{A}$, \begin{eqnarray*}D(\alpha) = \sum_{i=1}^{n} a_{i} D(\alpha_{i}) = \sum_{i=1}^{n} a_{i} \beta (\tau - \sigma)(\alpha_{i}) & = & \beta \left(\sum_{i=1}^{n} a_{i} \tau(\alpha_{i}) - \sum_{i=1}^{n} a_{i} \sigma(\alpha_{i})\right) \\ & = & \beta \left(\tau\left(\sum_{i=1}^{n} a_{i} \alpha_{i}\right) - \sigma\left(\sum_{i=1}^{n} a_{i} \alpha_{i}\right)\right) \\ & = & \beta (\tau - \sigma)(\alpha)\end{eqnarray*}

Since $\alpha \in \mathcal{A}$ is arbitrary, therefore, $D$ is inner.
\end{proof}

The theorem below now follows immediately.

\begin{theorem}\th\label{theorem 2.5}
Let $\mathcal{A}$ be of finite rank $n$ and $\{\alpha_{1}, \alpha_{2}, ..., \alpha_{n}\}$ be an $R$-basis of $\mathcal{A}$. Then a $(\sigma, \tau)$-derivation $D:\mathcal{A} \rightarrow \mathcal{A}$ is inner if and only if there exists some $\beta \in \mathcal{A}$ such that $D(\alpha_{i}) = \beta (\tau - \sigma)(\alpha_{i})$ for all $i \in \{1, 2, ..., n\}$.
\end{theorem}

\begin{lemma}\th\label{lemma 2.6}
Let $D: \mathcal{A} \rightarrow \mathcal{A}$ be a $(\sigma, \tau)$-derivation and $\alpha \in \mathcal{A}$. If there exists some $\beta \in \mathcal{A}$ such that $D(\alpha) = \beta (\tau - \sigma)(\alpha)$, then $$D(\alpha^{n}) = \beta (\tau - \sigma)(\alpha^{n})$$ for all $n \in \mathbb{N}$.
\end{lemma}

\begin{proof} We again use induction on $n$. If $n=1$, the result holds trivially. Now suppose that $D(\alpha^{k}) = \beta (\tau - \sigma)(\alpha^{k})$. Then
\begin{eqnarray*}
D(\alpha^{k+1}) & = & D(\alpha^{k}) \tau(\alpha) + \sigma(\alpha^{k}) D(\alpha) = \beta (\tau - \sigma)(\alpha^{k}) \tau(\alpha) + \sigma(\alpha^{k}) \beta (\tau - \sigma)(\alpha) \\ & = & \beta \left(\tau(\alpha^{k+1}) - \sigma(\alpha^{k}) \tau(\alpha) + \sigma(\alpha^{k}) \tau(\alpha) - \sigma(\alpha^{k+1})\right) \\ & = & \beta (\tau - \sigma)(\alpha^{k+1})
\end{eqnarray*}
Induction is complete and so is proof.
\end{proof}

\begin{lemma}\th\label{lemma 2.7}
Let $\mathcal{A}$ be of finite rank $n$ and suppose that $\mathcal{A}$ has an $R$-basis of the form $\{1, \alpha, \alpha^{2}, ..., \alpha^{n-1}\}$ for some $\alpha \in \mathcal{A}$. Let $D: \mathcal{A} \rightarrow \mathcal{A}$ be a $(\sigma, \tau)$-derivation. If there exists some $\beta \in \mathcal{A}$ such that $D(\alpha) = \beta (\tau - \sigma)(\alpha)$, then $D$ is inner.
\end{lemma}

\begin{proof} Since $D(1) = 0$ and $\sigma(1) = \tau(1) = 1$, so $D(1) = \beta (\tau - \sigma)(1)$. Further since $D(\alpha) = \beta (\tau - \sigma)(\alpha)$, therefore, by \th\ref{lemma 2.6}, $$D(\alpha^{k}) = \beta (\tau - \sigma)(\alpha^{k})$$ for all $k \in \mathbb{N}$. So $D(\alpha^{k}) = \beta (\tau - \sigma)(\alpha^{k})$ for every $k \in \{0, 1, ..., n-1\}$. Now we use \th\ref{lemma 2.4}.
\end{proof}

The theorem below follows immediately.

\begin{theorem}\th\label{theorem 2.8}
Let $\mathcal{A}$ be of finite rank $n$ and suppose that $\mathcal{A}$ has an $R$-basis of the form $\{1, \alpha, \alpha^{2}, ..., \alpha^{n-1}\}$ for some $\alpha \in \mathcal{A}$. Then a $(\sigma, \tau)$-derivation $D:\mathcal{A} \rightarrow \mathcal{A}$ is inner if and only if there exists some $\beta \in \mathcal{A}$ such that $D(\alpha) = \beta (\tau - \sigma)(\alpha)$.
\end{theorem}

\section{\texorpdfstring{$(\sigma, \tau)$}{Lg}-Derivations of Number Rings}\label{section 3}
We know by \th\ref{theorem 1.1} that if $K$ is a number field, then $K = \mathbb{Q}(\theta)$ for some algebraic integer $\theta$. \th\ref{theorem 2.3} gives the lemma below:
\begin{lemma}\th\label{lemma 3.1}
Let $K = \mathbb{Q}(\theta)$ ($\theta$ an algebraic integer) be a number field of degree $n$ such that $\{1, \theta, ..., \theta^{n-1}\}$ is a power basis of $K$. Let $\sigma$ and $\tau$ be two different non-zero $\mathbb{Z}$-algebra endomorphisms of $O_{K}$. If a $\mathbb{Z}$-linear map $D:O_{K} \rightarrow O_{K}$ is a $(\sigma, \tau)$-derivation, then $$D(\theta^{k}) = \left(\sum_{(i,j) \in S_{k-1}}  \sigma(\theta^{i}) \tau(\theta^{j})\right)D(\theta)$$ for all $k \in \{1, 2, ..., n-1\}$.
\end{lemma}

\subsection{\texorpdfstring{$(\sigma, \tau)$}{Lg}-Derivations of Quadratic Number Rings}\label{subsection 3.1}
By \th\ref{theorem 1.3}, if $K$ is a quadratic field, then $K = \mathbb{Q}(\sqrt{d})$ for some square-free rational integer $d$. In this section, we take $K = \mathbb{Q}(\sqrt{d})$ and $\sigma$, $\tau$ as two different non-zero ring endomorphisms of $O_{K}$. Note that any ring endomorphism of the ring $O_{K}$ is a unital $\mathbb{Z}$-algebra endomorphism of the unital $\mathbb{Z}$-algebra $O_{K}$. As an application of \th\ref{lemma 2.1}, we present an independent proof of a result of \cite{Chaudhuri} stated as follows. 

\begin{theorem}[{\cite[Theorem 4.2]{Chaudhuri}}] \th\label{theorem 3.2}
Every $\mathbb{Z}$-linear map $D:O_{K} \rightarrow O_{K}$ with $D(1)=0$ is a $(\sigma, \tau)$-derivation of $O_{K}$.
\end{theorem}

\begin{proof} Only two possibilities arise: $d \not\equiv 1 \hspace{0.1cm} (\text{mod} \hspace{0.1cm} 4)$ and $d \equiv 1 \hspace{0.1cm} (\text{mod} \hspace{0.1cm} 4)$. Since $(\sqrt{d})^{2} = d$, so if $\theta$ is a non-zero ring endomorphism of $O_{K}$, then $(\theta(\sqrt{d}))^{2} = d$ or that $\theta(\sqrt{d}) = \pm \sqrt{d}$.

When $d \not\equiv 1 \hspace{0.1cm} (\text{mod} \hspace{0.1cm} 4)$, then from \th\ref{theorem 1.4}, $O_{K} = \mathbb{Z}[\sqrt{d}]$ and from \th\ref{theorem 1.5}, $\{1, \sqrt{d}\}$ is a $\mathbb{Z}$-basis of $O_{K}$. In this case, exactly two possibilities arise. For every $a, b \in \mathbb{Z}$,
\begin{itemize}
\item[(a)] $\sigma(a + b \sqrt{d}) = a + b \sqrt{d}$; $\tau(a + b \sqrt{d}) = a - b \sqrt{d}$,
\item[(b)] $\sigma(a + b \sqrt{d}) = a - b \sqrt{d}$; $\tau(a + b \sqrt{d}) = a + b \sqrt{d}.$
\end{itemize} 

When $d \equiv 1 \hspace{0.1cm} (\text{mod} \hspace{0.1cm} 4)$, then by \th\ref{theorem 1.4}, $O_{K} = \mathbb{Z}[\frac{1+\sqrt{d}}{2}]$ and by \th\ref{theorem 1.5}, $\{1,\frac{1+\sqrt{d}}{2}\}$ is a $\mathbb{Z}$-basis of $O_{K}$. Thus, in this case, too, exactly two possibilities arise. For all $a, b \in \mathbb{Z}$,
\begin{itemize}
\item[(a)] $\sigma(a + b (\frac{1+\sqrt{d}}{2})) = a + b (\frac{1+\sqrt{d}}{2})$; $\tau(a + b (\frac{1+\sqrt{d}}{2})) = a + b (\frac{1-\sqrt{d}}{2})$,
\item[(b)] $\sigma(a + b (\frac{1+\sqrt{d})}{2})) = a + b (\frac{1-\sqrt{d}}{2})$; $\tau(a + b (\frac{1+\sqrt{d}}{2})) = a + b (\frac{1+\sqrt{d}}{2})$.
\end{itemize} 

Now the proof trivially follows from \th\ref{lemma 2.1}. 
\end{proof}

Further, in the same paper \cite{Chaudhuri}, the author has proved the theorem stated below.

\begin{theorem}[{\cite[Theorem 4.2]{Chaudhuri}}] \th\label{theorem 3.3}
Let $d \not\equiv 1 \hspace{0.1cm} (\text{mod} \hspace{0.1cm} 4)$. Let $D:O_{K} \rightarrow O_{K}$ be a $(\sigma, \tau)$-derivation and $D(\sqrt{d}) = c_{0} + c_{1} \sqrt{d}$ for some $c_{0}, c_{1} \in \mathbb{Z}$. If $2 d$ divides $c_{0}$ and $c_{1}$ is even, then $D$ is inner.
\end{theorem}

We present a different proof of this theorem using \th\ref{theorem 2.8} and also prove that the conditions in the above theorem are necessary.

\begin{theorem}\th\label{theorem 3.4}
Let $d \not\equiv 1 \hspace{0.1cm} (\text{mod} \hspace{0.1cm} 4)$. Let $D:O_{K} \rightarrow O_{K}$ be a $(\sigma, \tau)$-derivation and $D(\sqrt{d}) = c_{0} + c_{1} \sqrt{d}$ for some $c_{0}, c_{1} \in \mathbb{Z}$. Then $D$ is inner if and only if $c_{0}$ is divisible by $2 d$ and $c_{1}$ is even. In particular, the following conditions hold:
\begin{enumerate}
\item[(i)] $D$ is inner if $2d$ divides both $c_{0}$ and $c_{1}$.
\item[(ii)] $O_{K}$ has non-trivial outer $(\sigma, \tau)$-derivations, that is, $\text{Out}_{(\sigma, \tau)}(O_{K}) \neq \{0\}$.
\end{enumerate}
\end{theorem}

\begin{proof}
Since $d \not\equiv 1 \hspace{0.1cm} (\text{mod} \hspace{0.1cm} 4)$, so $\{1, \sqrt{d}\}$ is a basis of the $\mathbb{Z}$-module $O_{K} = \mathbb{\mathbb{Z}}[\sqrt{d}]$.

$O_{K} = \mathbb{\mathbb{Z}}[\sqrt{d}]$ has precisely two different non-zero ring endomorphisms: $\phi_{1}(a + b \sqrt{d}) \\ = a + b \sqrt{d}$ and $\phi_{2}(a + b \sqrt{d}) = a - b \sqrt{d}$ ($a, b \in \mathbb{Z}$). Therefore, $(\sigma, \tau) = (\phi_{1}, \phi_{2})$ or $(\phi_{2}, \phi_{1})$. We prove the result for $(\sigma, \tau) = (\phi_{2}, \phi_{1})$ since the other follows similarly.

First, let $D$ be inner. Then there exists some $\beta = b_{0} + b_{1} \sqrt{d}$ for some $b_{0}, b_{1} \in \mathbb{Z}$ such that $D(\sqrt{d}) = \beta (\tau - \sigma)(\sqrt{d})$. Therefore, \begin{eqnarray*}
c_{0} + c_{1} \sqrt{d} = D(\sqrt{d}) & = & \beta (\tau - \sigma)(\sqrt{d}) \\ & = & (b_{0} + b_{1} \sqrt{d}) (\tau(\sqrt{d}) - \sigma(\sqrt{d})) \\ & = & 2d b_{1} + 2b_{0} \sqrt{d}
\end{eqnarray*}
Since $\{1, \sqrt{d}\}$ is a basis of the $\mathbb{Z}$-module $O_{K} = \mathbb{\mathbb{Z}}[\sqrt{d}]$, therefore, $$2d b_{1} = c_{0} \hspace{0.2cm} \text{and} \hspace{0.2cm} 2 b_{0} = c_{1}.$$
Obviously, $b_{0} = \frac{c_{1}}{2}$ and $b_{1} = \frac{c_{0}}{2d}$. Since $b_{0}, b_{1} \in \mathbb{Z}$, therefore, $c_{1}$ is divisible by $2$ and $c_{0}$ is divisible by $2d$. 

Conversely, let $c_{0}$ be divisible by $2d$ and $c_{1}$ be even.
Define $\beta = \left(\frac{c_{0}}{2d}\right) + \left(\frac{c_{1}}{2} \right) \sqrt{d}$. Then it can be verified that $\beta (\tau - \sigma)(\sqrt{d}) = D(\sqrt{d})$ as $D(\sqrt{d}) = c_{0} + c_{1} \sqrt{d}$.
Now the result follows immediately from \th\ref{theorem 2.8}.

Finally, (i) follows immediately and (ii) follows from the above proved characterization of inner derivations of $O_{K}$ together with  Lemma \ref{lemma 2.5} and Remark \ref{remark 2.6}.
\end{proof}

\begin{theorem}
Let $d \equiv 1 \hspace{0.1cm} (\text{mod} \hspace{0.1cm} 4)$. Let $D:O_{K} \rightarrow O_{K}$ be a $(\sigma, \tau)$-derivation and $D(\sqrt{d}) = c_{0} + c_{1} (\frac{1+\sqrt{d}}{2})$ for some $c_{0}, c_{1} \in \mathbb{Z}$. Then $D$ is inner if and only if $d$ divides $-c_{0} + c_{1} \left(\frac{d-1}{2}\right)$ and $2 c_{0} + c_{1}$. In particular, the following conditions hold:
\begin{enumerate}
\item[(i)] $D$ is inner if $d$ divides both $c_{0}$ and $c_{1}$.
\item[(ii)] $O_{K}$ has non-trivial outer $(\sigma, \tau)$-derivations, that is, $\text{Out}_{(\sigma, \tau)}(O_{K}) \neq \{0\}$. 
\end{enumerate}      
\end{theorem}

\begin{proof}
Since $d \equiv 1 \hspace{0.1cm} (\text{mod} \hspace{0.1cm} 4)$, so $\{1, \frac{1+\sqrt{d}}{2}\}$ is a basis of the $\mathbb{Z}$-module $O_{K} = \mathbb{\mathbb{Z}}[\frac{1+\sqrt{d}}{2}]$.
$O_{K}$ has precisely two different non-zero ring endomorphisms, namely, $\phi_{1}(a + b (\frac{1+\sqrt{d}}{2})) = a + b (\frac{1+\sqrt{d}}{2})$ and $\phi_{2}(a + b (\frac{1+\sqrt{d}}{2})) = a + b (\frac{1-\sqrt{d}}{2})$ ($a, b \in \mathbb{Z}$). Therefore, $(\sigma, \tau) = (\phi_{1}, \phi_{2})$ or $(\phi_{2}, \phi_{1})$. We prove the result for $(\sigma, \tau) = (\phi_{2}, \phi_{1})$ since the other follows similarly.

First, let $D$ be inner. Then there exists some $\beta = b_{0} + b_{1} (\frac{1+\sqrt{d}}{2})$ for some $b_{0}, b_{1} \in \mathbb{Z}$ such that $D(\frac{1+\sqrt{d}}{2}) = \beta (\tau - \sigma)(\frac{1+\sqrt{d}}{2})$. Therefore, 

\begin{equation*}
\begin{aligned}
c_{0} + c_{1} \left(\frac{1+\sqrt{d}}{2}\right) & = D\left(\frac{1+\sqrt{d}}{2}\right) = \beta \left(\tau - \sigma \right)\left(\frac{1+\sqrt{d}}{2}\right) \\ & = \left(b_{0} + b_{1} \left(\frac{1+\sqrt{d}}{2}\right)\right) \left(\tau \left(\frac{1+\sqrt{d}}{2} \right) - \sigma \left(\frac{1+\sqrt{d}}{2}\right) \right) \\ & = \left(b_{0} + b_{1} \left(\frac{1+\sqrt{d}}{2} \right) \right) \left( \left(\frac{1+\sqrt{d}}{2}\right) - \left(\frac{1-\sqrt{d}}{2}\right) \right) \\ & = \left(-b_{0} + b_{1} \left( \frac{d-1}{2} \right)\right) + (2b_{0} + b_{1}) \left(\frac{1+\sqrt{d}}{2}\right)
\end{aligned}
\end{equation*}
Since $\{1, \frac{1+ \sqrt{d}}{2}\}$ is a basis of $O_{K} = \mathbb{\mathbb{Z}}[\frac{1+ \sqrt{d}}{2}]$, therefore, $$-b_{0} + b_{1} \left( \frac{d-1}{2} \right) = c_{0} \hspace{0.2cm} \text{and} \hspace{0.2cm} 2b_{0} + b_{1} = c_{1}.$$
Note that since $d \equiv 1 \hspace{0.1cm} (\text{mod} \hspace{0.1cm} 4)$, so $\frac{d-1}{2} \in \mathbb{Z}$. Solving, we get, $$b_{0} = \frac{1}{d} \left(- c_{0} + c_{1} \left(\frac{d-1}{2} \right)\right) \hspace{0.1cm} \text{and} \hspace{0.1cm} b_{1} = \frac{1}{d} \left(2 c_{0} + c_{1}\right).$$ Since $b_{0}, b_{1} \in \mathbb{Z}$, therefore, $d$ divides $- c_{0} + c_{1} \left(\frac{d-1}{2} \right)$ and $2 c_{0} + c_{1}$. 

Conversely, let $d$ divide $- c_{0} + c_{1} \left(\frac{d-1}{2} \right)$ and $2 c_{0}+c_{1}$. Define $$\beta = \left(\frac{1}{d} \left(- c_{0} + c_{1} \left(\frac{d-1}{2} \right)\right)\right) + \left(\frac{1}{d} \left(2 c_{0} + c_{1}\right)\right)\left(\frac{1+\sqrt{d}}{2}\right).$$
Then it can be verified that $\beta \left(\tau - \sigma \right)\left(\frac{1+\sqrt{d}}{2}\right) = D\left(\frac{1+\sqrt{d}}{2}\right)$ as $D\left(\frac{1+\sqrt{d}}{2}\right) = c_{0} + c_{1} \left(\frac{1+\sqrt{d}}{2}\right)$.
The result now follows immediately by \th\ref{theorem 2.8}.

Finally, (i) follows immediately and (ii) follows from the above proved characterization of inner derivations of $O_{K}$ together with  Lemma \ref{lemma 2.5} and Remark \ref{remark 2.6}.
\end{proof}

\subsection{\texorpdfstring{$(\sigma, \tau)$}{Lg}-Derivations of Cyclotomic Number Rings}\label{subsection 3.2}
In this section, $K = \mathbb{Q}(\zeta)$ denotes a $p^{\text{th}}$ cyclotomic field, where $p$ is an odd rational prime. By \th\ref{theorem 1.7}, the ring of algebraic integers of $K$ is $O_{K} = \mathbb{Z}[\zeta]$ and $\{1, \zeta, \zeta^{2}, ..., \zeta^{p-2}\}$ is an integral basis of $K$. Further, $\sigma$ and $\tau$ denote any two different non-zero ring endomorphisms of $O_{K} = \mathbb{Z}[\zeta]$.

\begin{lemma}\th\label{lemma 3.6}
Let for each $k \in \{0, 1, ..., p-2\}$, $S_{k}$'s be sets as defined in Section \ref{subsection 2.3}. Then $$\sum_{(i,j) \in \cup_{k=0}^{p-2} S_{k}} \sigma(\zeta^{i}) \tau(\zeta^{j}) = 0.$$ 
\end{lemma}

\begin{proof} All the non-zero ring endomorphisms $\phi_{i}$ of $O_{K}$ are given by $\phi_{i}(\zeta) = \zeta^{i},$ where $i \in \{1, 2, ..., p-1\}$. Since $\sigma$ and $\tau$ are non-zero and different, therefore, $$\sigma(\zeta) = \zeta^{u} \hspace{0.2cm} \text{and} \hspace{0.2cm} \tau(\zeta) = \zeta^{u+v}$$ for some $u, v \in \{1, 2, ..., p-1\}$ such that $u+v \neq p$. Now, $$\sum_{(i,j) \in \cup_{k=0}^{p-2} S_{k}} \sigma(\zeta^{i}) \tau(\zeta^{j}) = \sum_{k=0}^{p-2} \sum_{(i,j) \in S_{k}} \sigma(\zeta^{i}) \tau(\zeta^{j}).$$
For any $k \in \{0, 1, ..., p-2\}$, $$\sum_{(i,j) \in S_{k}} \sigma(\zeta^{i}) \tau(\zeta^{j}) = \sum_{(i,j) \in S_{k}} \zeta^{ku} \zeta^{jv}.$$

There are three observations for each $k \in \{0, 1, ..., p-2\}$:
\begin{enumerate}
\item[(a)] The sum $\sum_{(i,j) \in S_{k}} \sigma(\zeta^{i}) \tau(\zeta^{j})$ contains $k+1$ terms and all these $k+1$ terms are different from each other since $u+v \neq p$.

\item[(b)] The sum $\sum_{(i,j) \in \cup_{k=0}^{p-2} S_{k}} \sigma(\zeta^{i}) \tau(\zeta^{j})$ contains exactly $\frac{p(p-1)}{2}$ terms.

\item[(c)] Since the prime $p$ is odd, therefore, for each $k \in \{1, 2, ..., p-1\}$, $\zeta^{k}$ is a primitive $p^{\text{th}}$ root of unity.
\end{enumerate}

These together imply that each term from the set $\{1, \zeta, \zeta^{2}, ..., \zeta^{p-1}\}$ is being repeated $\frac{p-1}{2}$ times in the sum $\sum_{(i,j) \in \cup_{k=0}^{p-2} S_{k}} \sigma(\zeta^{i}) \tau(\zeta^{j})$. Hence, $$\sum_{(i,j) \in \cup_{m=0}^{p-2} S_{m}} \sigma(\zeta^{i}) \tau(\zeta^{j}) = \frac{p-1}{2} \left(1 + \zeta + \zeta^{2} + ... + \zeta^{p-1}\right) = 0,$$ since $1 + \zeta + \zeta^{2} + ... + \zeta^{p-1} = 0$ by \th\ref{theorem 1.6}.

This proves the required result. Note that the proof also follows directly using the direct formula for the sum of the first $n$ ($n \in \mathbb{N}$) of a geometric sequence.
\end{proof}

\begin{lemma}\th\label{lemma 3.7}
Let $D:O_{K} \rightarrow O_{K}$ be a $\mathbb{Z}$-linear map with $D(1) = 0$ and \begin{equation}\label{eq 3.1}D(\zeta^{k}) = \left(\sum_{(i,j) \in S_{k-1}}  \sigma(\zeta^{i}) \tau(\zeta^{j})\right)D(\zeta)\end{equation} for all $k \in \{1, 2, ..., p-2\}$. Then $D$ is a $(\sigma, \tau)$-derivation.
\end{lemma}

\begin{proof} According to \th\ref{lemma 2.1}, the result can be concluded by establishing that \begin{equation}\label{eq 3.2}D(\zeta^{i+j}) = D(\zeta^{i}) \tau(\zeta^{j}) + \sigma(\zeta^{i}) D(\zeta^{j})\end{equation} for all $i, j \in \{0, 1, ..., p-2\}$.

The relations (\ref{eq 3.2}) hold trivially when atleast one of $i$ or $j$ is $0$, using the fact that $\sigma(1) = \tau(1) = 1$ and $D(1) = 0$. So now let $i, j \in \{1, 2, ..., p-2\}$. Using (\ref{eq 3.1}), we get:
\begin{eqnarray*}
D(\zeta^{i}) \tau(\zeta^{j}) = \left(\sum_{(s,t) \in S_{i-1}}  \sigma(\zeta^{s}) \tau(\zeta^{t})\right)D(\zeta) \tau(\zeta^{j}) & = & \left(\sum_{(s,t) \in S_{i-1}}  \sigma(\zeta^{s}) \tau(\zeta^{t+j})\right)D(\zeta) \\ & = & \left(\sum_{s=0}^{i-1} \sigma(\zeta^{s}) \tau(\zeta^{i-1-s+j})\right)D(\zeta)\end{eqnarray*} and \begin{eqnarray*}
\sigma(\zeta^{i}) D(\zeta^{j}) = \sigma(\zeta^{i})\left(\sum_{(s,t) \in S_{j-1}} \sigma(\zeta^{s}) \tau(\zeta^{t})\right)D(\zeta) & = & \left(\sum_{(s,t) \in S_{j-1}} \sigma(\zeta^{i+s}) \tau(\zeta^{t})\right)D(\zeta) \\ & = & \left(\sum_{s=0}^{j-1} \sigma(\zeta^{i+s}) \tau(\zeta^{j-1-s})\right)D(\zeta) \\ & = & \left(\sum_{s=i}^{i+j-1} \sigma(\zeta^{s}) \tau(\zeta^{j-1+i-s})\right)D(\zeta).
\end{eqnarray*}
Therefore, \begin{equation}\label{eq 3.3}
D(\zeta^{i}) \tau(\zeta^{j}) + \sigma(\zeta^{i}) D(\zeta^{j}) = \left(\sum_{(s,t) \in S_{i+j-1}} \sigma(\zeta^{s}) \tau(\zeta^{t})\right)D(\zeta).
\end{equation}

Since $i, j \in \{1, 2,..., p-2\}$, so $i+j \in \{2, 3 ..., 2(p-2)\}$. We partition the proof into the following four cases.\vspace{10pt}

\textbf{Case 1:} $i+j \leq p-2$.

Since $i+j \leqslant p-2$, so by (\ref{eq 3.1}), $D(\zeta^{i+j}) = \left(\sum_{(s,t) \in S_{i+j-1}} \sigma(\zeta^{s}) \tau(\zeta^{t})\right)D(\zeta)$. Then by (\ref{eq 3.3}), $D(\zeta^{i}) \tau(\zeta^{j}) + \sigma(\zeta^{i}) D(\zeta^{j}) = D(\zeta^{i+j})$. Therefore, in this case, relations (\ref{eq 3.2}) hold.\vspace{10pt}

\textbf{Case 2:} $i+j = p-1$.

By \th\ref{lemma 3.6}, $\sum_{(i,j) \in \cup_{k=0}^{p-2} S_{k}} \sigma(\zeta^{i}) \tau(\zeta^{j}) = 0$. 

$\Rightarrow \sum_{k=0}^{p-2} \sum_{(i,j) \in S_{k}} \sigma(\zeta^{i}) \tau(\zeta^{j}) = 0$.

$\Rightarrow \sum_{(i,j) \in S_{p-2}} \sigma(\zeta^{i}) \tau(\zeta^{j}) = - \left(\sum_{k=0}^{p-3} \sum_{(i,j) \in S_{k}} \sigma(\zeta^{i}) \tau(\zeta^{j})\right)$.

$\Rightarrow \left(\sum_{(i,j) \in S_{p-2}} \sigma(\zeta^{i}) \tau(\zeta^{j})\right)D(\zeta) = - \left(\sum_{k=0}^{p-3} \left(\sum_{(i,j) \in S_{k}} \sigma(\zeta^{i}) \tau(\zeta^{j})\right)D(\zeta)\right)$.

Therefore, using (\ref{eq 3.1}) and (\ref{eq 3.3}), $$D(\zeta^{i}) \tau(\zeta^{j}) + \sigma(\zeta^{i}) D(\zeta^{j}) = -\left(\sum_{k=0}^{p-3} D(\zeta^{k+1})\right) = -D(1 + \zeta + \zeta^{2} + ... + \zeta^{p-2}) = D(\zeta^{p-1}).$$

So (\ref{eq 3.2}) holds in this case too.\vspace{10pt}

\textbf{Case 3:} $i+j = p$.

Since the ring endomorphisms $\sigma$ and $\tau$ are non-zero and different, therefore, $\sigma(\zeta) = \zeta^{u} \hspace{0.2cm} \text{and} \hspace{0.2cm} \tau(\zeta) = \zeta^{u+v}$ for some $u, v \in \{1, 2, ..., p-1\}$ such that $u+v \neq p$. Now by (\ref{eq 3.3}),
\begin{equation*}
\begin{aligned}
D(\zeta^{i}) \tau(\zeta^{j}) + \sigma(\zeta^{i}) D(\zeta^{j}) & = \left(\sum_{(s,t) \in S_{i+j-1}} \sigma(\zeta^{s}) \tau(\zeta^{t})\right)D(\zeta) 
 = \left(\sum_{(s,t) \in S_{p-1}} \sigma(\zeta^{s}) \tau(\zeta^{t})\right)D(\zeta) 
\\ & = \left(\sum_{s=0}^{p-1} \sigma(\zeta^{s}) \tau(\zeta^{p-1-s})\right)D(\zeta) 
 = \zeta^{(p-1)u} \left(\sum_{s=0}^{p-1} \zeta^{(p-1-s)v})\right)D(\zeta) 
\\ & = \zeta^{(p-1)u} \left(1 + \zeta^{v} + \zeta^{2v} + ... + \zeta^{(p-1)v}\right)D(\zeta) 
 = 0,
\end{aligned}
\end{equation*} as $\zeta^{v}$ satisfies the cyclotomic polynomial. Therefore, $D(\zeta^{i}) \tau(\zeta^{j}) + \sigma(\zeta^{i}) D(\zeta^{j}) = D(\zeta^{p})$ as $\zeta^{p} = 1$ and $D(1) = 0$. So (\ref{eq 3.2}) holds in this case as well.\vspace{10pt}

\textbf{Case 4:} $p < i+j \leq 2(p-2)$.

Then $i+j = m + p$ for some $m \in \{1, 2, ..., p-2\}$. By (\ref{eq 3.1}),

\begin{eqnarray*}
D(\zeta^{i+j}) = D(\zeta^{m}) = \left(\sum_{(s,t) \in S_{m-1}} \sigma(\zeta^{s}) \tau(\zeta^{t})\right)D(\zeta) & = & \left(\sum_{t=0}^{m-1} \sigma(\zeta^{m-1-t}) \tau(\zeta^{t})\right)D(\zeta) \\ & = & \zeta^{(m-1)u} \left(\sum_{t=0}^{m-1} \zeta^{tv}\right)D(\zeta).\end{eqnarray*}

Further, by (\ref{eq 3.3}),  \begin{equation*}
\begin{aligned}
D(\zeta^{i}) \tau(\zeta^{j}) + \sigma(\zeta^{i}) D(\zeta^{j}) & = \left(\sum_{(s,t) \in S_{i+j-1}} \sigma(\zeta^{s}) \tau(\zeta^{t})\right)D(\zeta) = \left(\sum_{(s,t) \in S_{m+p-1}} \sigma(\zeta^{s}) \tau(\zeta^{t})\right)D(\zeta) 
\\ & = \left(\sum_{t=0}^{m+p-1} \sigma(\zeta^{m+p-1-t}) \tau(\zeta^{t})\right)D(\zeta) = \zeta^{(m+p-1)u} \left(\sum_{t=0}^{m+p-1} \zeta^{tv}\right)D(\zeta) 
\\ & = \zeta^{(m+p-1)u} \left(\sum_{t=0}^{m-1} \zeta^{tv} + \sum_{t=m}^{m+p-1} \zeta^{tv}\right)D(\zeta) 
\\ & = \zeta^{(m-1)u} \left(\sum_{t=0}^{m-1} \zeta^{tv} + \zeta^{mv} \left(\sum_{t=0}^{p-1} \zeta^{tv}\right)\right)D(\zeta) 
\\ & = \zeta^{(m-1)u} \left(\sum_{t=0}^{m-1} \zeta^{tv}\right)D(\zeta)
\end{aligned}
\end{equation*} since $\sum_{t=0}^{p-1} \zeta^{tv} = 0$ as $v \in \{1, 2, ..., p-1\}.$ Therefore, $D(\zeta^{i+j}) = D(\zeta^{i}) \tau(\zeta^{j}) + \sigma(\zeta^{i}) D(\zeta^{j})$. This gives again the relations (\ref{eq 3.2}). 

\noindent From the above four cases, w
e conclude that the relations (\ref{eq 3.2}) hold for all $i, j \in \{0, 1, ..., p-2\}$. The result now can be concluded from \th\ref{lemma 2.1}.
\end{proof}

As a consequence of \th\ref{lemma 3.1} and \th\ref{lemma 3.7}, we have the main theorem.

\begin{theorem}\th\label{theorem 3.8}
Let $D:O_{K} \rightarrow O_{K}$ be a $\mathbb{Z}$-linear map with $D(1) = 0$. Then $D$ is a $(\sigma, \tau)$-derivation if and only if the following relations hold for all $k \in \{1, 2, ..., p-2\}$: $$D(\zeta^{k}) = \left(\sum_{(i,j) \in S_{k-1}}  \sigma(\zeta^{i}) \tau(\zeta^{j})\right)D(\zeta).$$
\end{theorem}

\begin{corollary}\th\label{corollary 3.9}
The $\mathbb{Z}$-module $\mathcal{D}_{(\sigma, \tau)}(O_{K})$ is finitely generated of rank $p-1$.
\end{corollary}

\begin{proof} For every $i \in \{0, 1, ..., p-2\}$, define $D_{i}:O_{K} \rightarrow O_{K}$ as a $\mathbb{Z}$-linear map with $D(1) = 0$ and $$D_{i}(\zeta) = \zeta^{i}.$$ More precisely, for each $i \in \{0, 1, ..., p-2\}$, let $D_{i}:O_{K} \rightarrow O_{K}$ be a map defined by $$D_{i}\left(\sum_{j=0}^{p-2} a_{j} \zeta^{j}\right) = \sum_{j=1}^{p-2} a_{j} D_{i}(\zeta^{j}),$$ where for each $j \in \{1, 2, ..., p-2\}$, $D_{i}(\zeta^{j})$ is defined as $$D_{i}(\zeta^{j}) = \left(\sum_{(s,t) \in S_{j-1}}  \sigma(\zeta^{s}) \tau(\zeta^{t})\right)D_{i}(\zeta).$$
Then by \th\ref{theorem 3.8}, for each $i \in \{0, 1, ..., p-2\}$, $D_{i}:O_{K} \rightarrow O_{K}$ is a $(\sigma, \tau)$-derivation.
Further, it can be easily verified that $\{D_{0}, D_{1}, ..., D_{p-2}\}$ forms a linearly independent subset of the $\mathbb{Z}$-module $\mathcal{D}_{(\sigma, \tau)}(O_{K})$ that generates it. Hence, the result is proved.
\end{proof}

We observe from the preceding corollary that the rank of the $\mathbb{Z}$-module $\mathcal{D}_{(\sigma, \tau)}(O_{K})$ is equal to $p-1$, the degree of $K = \mathbb{Q}(\zeta)$ and which, by \th\ref{theorem 1.2}, is also equal to the rank of the free abelian (additive) group $O_{K} = \mathbb{Z}[\zeta]$.

Some immediate corollaries can be stated below.

\begin{corollary}\th\label{corollary 3.10}
Let $K = Q(\zeta)$ be a $3^{\text{th}}$ cyclotomic field. Then any $\mathbb{Z}$-linear map $D:O_{K} \rightarrow O_{K}$ with $D(1)=0$ is a $(\sigma, \tau)$-derivation.
\end{corollary}

\begin{corollary}\th\label{corollary 3.11}
There always exists a non-zero $(\sigma, \tau)$-derivation of $O_{K} = \mathbb{Z}[\zeta]$.
\end{corollary}

We propose the following conjecture.

\begin{conjecture}\th\label{conjecture 3.12}
Suppose $\beta = \sum_{i=0}^{p-2} b_{i} \zeta^{i} \in O_{K}$ and $\beta (\tau - \sigma) (\zeta) =  \sum_{i=0}^{p-2} \left( \sum_{j=0}^{p-2} a_{ij} b_{j} \right) \zeta^{i}$. Then $A = [a_{ij}]$ is a $(p-1) \times (p-1)$ matrix with determinant $p$.
\end{conjecture}

We have verified the conjecture for all odd primes less than $100$ using SageMath and MATLAB. Calculations for primes $3, 5, 7, 11, 13$ were done on SageMath (explicit matrix $A$ and $\text{det}(A)$ were found for all possible values of the pair $(\sigma(\zeta), \tau(\zeta))$ (see the appendix)). We then developed a code on MATLAB, and calculations for the remaining primes were done by running that MATLAB code. The appendix is available at the link \url{https://drive.google.com/file/d/19EMjw9bQUcUc6d6JeaitE4D7kysKQWtw/view?usp=sharing}. For larger primes, MATLAB took a long time to provide the output after running the code. Nevertheless, we believe the conjecture to be true for all odd primes. As a consequence, we propose another conjecture which too will hold once the above conjecture is proved.

Let $\beta = \sum_{i=0}^{p-2} b_{i} \zeta^{i} \in K$, $D$ be a $(\sigma, \tau)$-derivation of $O_{K}$ and $D(\zeta) = \sum_{i=0}^{p-2} c_{i} \zeta^{i} \in O_{K}$. Put $X^{T} = (b_{0} ~ b_{1} ~ ... ~ b_{p-2})$ and $C^{T} = (c_{0} ~ c_{1} ~ ... ~ c_{p-2})$. We denote by $\mathbb{Z}^{p-1}$ the set of all $(p-1) \times 1$ column matrices with entries from $\mathbb{Z}$. Also, $Adj(A)$ denotes the adjoint of the matrix $A$.

Note that since $\sigma \neq \tau$, therefore, $D(\zeta) = \beta(\tau - \sigma)(\zeta)$ always has a solution $\beta$ in $K$. Furthermore, $D(\zeta) = \beta(\tau - \sigma)(\zeta)$ has a solution $\beta$ in $O_{K}$ if and only if $AX = C$ has a solution in $\mathbb{Z}^{p-1}$.

\begin{conjecture}\th\label{conjecture 3.13}
With the above notations, $D$ is inner if and only if $\frac{1}{p}(Adj(A)C) \in \mathbb{Z}^{p-1}$. In particular, the following conditions hold:
\begin{enumerate}
\item[(i)] If $p$ divides $c_{i}$ for each $i \in \{0, 1, ..., p-2\}$, then $D$ is inner.
\item[(ii)] $O_{K}$ has non-trivial outer $(\sigma, \tau)$-derivations, that is, $\text{Out}_{(\sigma, \tau)}(O_{K}) \neq \{0\}$.
\end{enumerate}
\end{conjecture}

Given below are some interesting examples.

\begin{example} Let $p \geq 5$. Then not every $\mathbb{Z}$-linear map $D:\mathbb{Z}[\zeta] \longrightarrow \mathbb{Z}[\zeta]$ with $D(1)=0$ is a $(\sigma, \tau)$-derivation. In fact, there exists a non-zero $\mathbb{Z}$-linear map $D:\mathbb{Z}[\zeta] \longrightarrow \mathbb{Z}[\zeta]$ with $D(1)=0$ which is not a $(\sigma, \tau)$-derivation for every pair $(\sigma, \tau)$ ($\sigma \neq 0, \tau \neq 0, \sigma \neq \tau$). As an example, define $D:\mathbb{Z}[\zeta] \longrightarrow \mathbb{Z}[\zeta]$ by $D(\sum_{i=0}^{p-2} a_{i} \zeta^{i}) = a_{1} D(\zeta),$ where $D(\zeta) \in O_{K} \setminus \{0\}$, say, $D(\zeta) = \zeta$. Then $D$ is a non-zero well-defined $\mathbb{Z}$-linear map with $D(1) = 0.$ But in view of \th\ref{theorem 3.8}, $D$ is not a $(\sigma, \tau)$-derivation, since $$D(\zeta^{2}) \neq \left(\sum_{(i,j) \in S_{1}} \sigma(\zeta^{i}) \tau(\zeta^{j}) \right) D(\zeta).$$
\end{example}

\begin{example} Let $p = 5$. Let $\sigma$ and $\tau$ be given by $\sigma(\zeta) = \zeta$ and $\tau(\zeta) = \zeta^{2}$.

\textbf{(i)} Define $D:O_{K} \rightarrow O_{K}$ by $$D(\sum_{i=0}^{3} a_{i} \zeta^{i}) = \sum_{i=1}^{3} a_{i} D(\zeta^{i}), \hspace{0.2cm} \forall \hspace{0.2cm} a_{i} \in \mathbb{Z} \hspace{0.1cm} (i \in \{0, 1, 2, 3\}),$$ where for each $i \in \{1, 2, 3\}$, $D(\zeta^{i})$ is defined as $$D(\zeta^{i}) = \left(\sum_{(s,t) \in S_{i-1}}  \sigma(\zeta^{s}) \tau(\zeta^{t})\right)D(\zeta).$$

Note that $D$ is a $\mathbb{Z}$-linear map with $D(1) = 0$. Take $D(\zeta) \in O_{K} \setminus \{0\}$. Then by \th\ref{theorem 3.8}, $D$ is a non-zero $(\sigma, \tau)$-derivation. But as shown below, $D$ is not inner if, in particular, we take $D(\zeta) = \zeta$. Thus, in a cyclotomic field $K$, it is not necessary that every $(\sigma, \tau)$-derivation of its ring of algebraic integers $O_{K}$ is inner.

\textbf{(ii)} If possible, suppose that the $(\sigma, \tau)$-derivation $D:O_{K} \rightarrow O_{K}$ is inner. Then there exists some $\beta \in O_{K}$ such that $D(\alpha) = \beta(\sigma - \tau)(\alpha), \hspace{0.1cm} \forall \hspace{0.1cm} \alpha \in O_{K}$. Since $\beta \in O_{K}$, so $\beta = \sum_{i=0}^{3} b_{i} \zeta^{i}$ for some unique $b_{i} \in \mathbb{Z}$ for all $i \in \{0, 1, 2, 3\}$. 
\begin{eqnarray*}
\beta(\sigma - \tau)(\zeta) & = & \beta (\sigma(\zeta) - \tau(\zeta)) \\ & = & (b_{0} + b_{1} \zeta + b_{2} \zeta^{2} + b_{3} \zeta^{3})(\zeta - \zeta^{2}) \\ & = & (b_{2} - 2b_{3}) + (b_{0} + b_{2} - b_{3}) \zeta + (-b_{0} + b_{1} + b_{2} - b_{3}) \zeta^{2} + (-b_{1} + 2b_{2} - b_{3})\zeta^{3}
\end{eqnarray*}
Now $\beta(\sigma - \tau)(\zeta) = D(\zeta)$ with $D(\zeta) = \zeta$ gives $$(b_{2} - 2b_{3}) + (b_{0} + b_{2} - b_{3}) \zeta + (-b_{0} + b_{1} + b_{2} - b_{3}) \zeta^{2} + (-b_{1} + 2b_{2} - b_{3})\zeta^{3} = 0 + 1 \zeta + 0 \zeta^{2} + 0 \zeta^{3}.$$

\noindent $\Rightarrow b_{2} - 2b_{3} = 0$; $b_{0} + b_{2} - b_{3} = 1$; $-b_{0} + b_{1} + b_{2} - b_{3} = 0$; $-b_{1} + 2b_{2} - b_{3} = 0$.

\noindent This system of linear equations can be written as $AX = C$, where $$A = \begin{pmatrix}
0 & 0 & 1 & -2 \\
1 & 0 & 1 & -1 \\
-1 & 1 & 1 & -1 \\
0 & -1 & 2 & -1
\end{pmatrix}, ~ X^{T} = \begin{pmatrix}
b_{0} & b_{1} & b_{2} & b_{3} 
\end{pmatrix} ~ \text{and} ~ C^{T} = \begin{pmatrix}
0 & 1 & 0 & 0
\end{pmatrix}.$$
This can have a (unique) integral solution if and only if $A$ is a unimodular matrix, that is, the determinant of $A$ is either $1$ or $-1$. But the determinant of $A$ is $5$. Therefore, there is no integer solution. Hence, a contradiction.
Therefore, the given $(\sigma, \tau)$-derivation $D:O_{K} \rightarrow O_{K}$ is not inner for $D(\zeta) = \zeta$.
\end{example}

\subsection{\texorpdfstring{$(\sigma, \tau)$}{Lg}-Derivations of Bi-quadratic Number Rings}\label{subsection 3.3}
Here we present an elegant application of \th\ref{lemma 2.1} in classifying all $(\sigma, \tau)$-derivations and inner $(\sigma, \tau)$-derivations of the bi-quadratic number ring $\mathbb{Z}[\sqrt{m}, \sqrt{n}]$. Throughout this subsection, $K$ denotes the quartic number field $K = \mathbb{Q}(\sqrt{m}, \sqrt{n})$, where $m$ and $n$ are distinct square-free rational integers, $\mathcal{A}$ the number ring $\mathcal{A} = \mathbb{Z}[\sqrt{m}, \sqrt{n}]$ with $\mathbb{Z}$-basis $\{1, \sqrt{m}, \sqrt{n}, \sqrt{mn}\}$, and $\sigma$ and $\tau$ any two non-zero ring endomorphisms of $\mathcal{A}$.

\begin{lemma}\th\label{lemma 3.16}
If $D:\mathcal{A} \rightarrow \mathcal{A}$ is a non-zero $(\sigma, \tau)$-derivation of $\mathcal{A}$, then precisely one case amongst (i), (ii), or (iii) holds.
\begin{enumerate}
\item[(i)] $D(\sqrt{m}) = 0$, $D(\sqrt{n}) \neq 0$, $D(\sqrt{mn}) \neq 0$ and $D(\sqrt{mn}) = \pm \sqrt{m} D(\sqrt{n})$. Further, in this case, the possible values of the pair $(\sigma, \tau)$ are $(\phi_{1}, \phi_{2}), (\phi_{2}, \phi_{1}), (\phi_{3}, \phi_{4}), (\phi_{4}, \phi_{3})$. 

In fact, $D(\sqrt{mn}) = \sqrt{m} D(\sqrt{n})$ if $(\sigma, \tau) = (\phi_{1}, \phi_{2})$ or $(\phi_{2}, \phi_{1})$ and $D(\sqrt{mn}) \\ = - \sqrt{m} D(\sqrt{n})$ if $(\sigma, \tau) = (\phi_{3}, \phi_{4})$ or $(\phi_{4}, \phi_{3})$.\vspace{0.2cm}

\item[(ii)] $D(\sqrt{m}) \neq 0$, $D(\sqrt{n}) = 0$, $D(\sqrt{mn}) \neq 0$ and $D(\sqrt{mn}) = \pm \sqrt{n} D(\sqrt{m})$. Further, in this case, the possible values of the pair $(\sigma, \tau)$ are $(\phi_{1}, \phi_{3}), (\phi_{2}, \phi_{4}), (\phi_{3}, \phi_{1}), (\phi_{4}, \phi_{2})$.

In fact, $D(\sqrt{mn}) = \sqrt{n} D(\sqrt{m})$ if $(\sigma, \tau) = (\phi_{1}, \phi_{3})$ or $(\phi_{3}, \phi_{1})$ and $D(\sqrt{mn}) \\ = - \sqrt{n} D(\sqrt{m})$ if $(\sigma, \tau) = (\phi_{2}, \phi_{4})$ or $(\phi_{4}, \phi_{2})$.\vspace{0.2cm}

\item[(iii)] $D(\sqrt{m}) \neq 0$, $D(\sqrt{n}) \neq 0$, $D(\sqrt{mn}) = 0$ and $\sqrt{m} D(\sqrt{n}) = \pm \sqrt{n} D(\sqrt{m})$. Further, in this case, the possible values of the pair $(\sigma, \tau)$ are $(\phi_{1}, \phi_{4}), (\phi_{2}, \phi_{3}), (\phi_{4}, \phi_{1}), (\phi_{3}, \phi_{2})$.

In fact, $\sqrt{m} D(\sqrt{n}) = \sqrt{n} D(\sqrt{m})$ if $(\sigma, \tau) = (\phi_{1}, \phi_{4})$ or $(\phi_{4}, \phi_{1})$ and $\sqrt{m} D(\sqrt{n}) \\ = - \sqrt{n} D(\sqrt{m})$ if $(\sigma, \tau) = (\phi_{2}, \phi_{3})$ or $(\phi_{3}, \phi_{2})$.
\end{enumerate} 

Here $\phi_{1}, \phi_{2}, \phi_{3}, \phi_{4}$ are the non-zero ring endomorphisms of $\mathcal{A} = \mathbb{Z}[\sqrt{m}, \sqrt{n}]$ given by 
\begin{enumerate}[noitemsep]
\item[(a)] $\phi_{1}(\sqrt{m}) = \sqrt{m}$, $\phi_{1}(\sqrt{n}) = \sqrt{n}$.

\item[(b)] $\phi_{2}(\sqrt{m}) = \sqrt{m}$, $\phi_{2}(\sqrt{n}) = - \sqrt{n}$.

\item[(c)] $\phi_{3}(\sqrt{m}) = - \sqrt{m}$, $\phi_{3}(\sqrt{n}) = \sqrt{n}$.

\item[(d)] $\phi_{4}(\sqrt{m}) = - \sqrt{m}$, $\phi_{4}(\sqrt{n}) = - \sqrt{n}$.
\end{enumerate}
\end{lemma}

\begin{proof}
If $\phi$ is a non-zero ring endomorphisms of $\mathcal{A}$, then $(\phi(\sqrt{m}))^{2} = m \hspace{0.2cm} \text{and} \hspace{0.2cm} (\phi(\sqrt{n}))^{2} = n$. So $\phi(\sqrt{m}) = \pm \sqrt{m}$ and $\phi(\sqrt{n}) = \pm \sqrt{n}.$ Therefore, $\mathcal{A}$ has precisely the four non-zero ring endomorphisms $\phi_{1}, \phi_{2}, \phi_{3},  \phi_{4}$ given in the table below. 

\begin{table}[H] 
\caption{}
\centering
\begin{tabular}{|c||c|c|c|c|}
\hline 
 & $1$ & $\sqrt{m}$ & $\sqrt{n}$ & $\sqrt{mn}$ \\ 
\hline \hline
$\phi_{1}$ & $1$ & $\sqrt{m}$ & $\sqrt{n}$ & $\sqrt{mn}$ \\ 
\hline 
$\phi_{2}$ & $1$ & $\sqrt{m}$ & $- \sqrt{n}$ & $- \sqrt{mn}$ \\ 
\hline 
$\phi_{3}$ & $1$ & $- \sqrt{m}$ & $\sqrt{n}$ & $- \sqrt{mn}$ \\ 
\hline 
$\phi_{4}$ & $1$ & $- \sqrt{m}$ & $- \sqrt{n}$ & $\sqrt{mn}$ \\ 
\hline 
\end{tabular}
\label{table 1}
\end{table}

Put $\alpha_{1} = 1$, $\alpha_{2} = \sqrt{m}$, $\alpha_{3} = \sqrt{n}$ and $\alpha_{4} = \sqrt{mn}$. Then by \th\ref{lemma 2.1}, $D(\alpha_{i} \alpha_{j}) = D(\alpha_{i}) \tau(\alpha_{j}) + \sigma(\alpha_{i}) D(\alpha_{j}), \hspace{0.1cm} \forall \hspace{0.1cm} i, j \in \{1, 2, 3, 4\}$. Also, $D$ is $\mathbb{Z}$-linear. So we get the following table \ref{table 2}.

\begin{table}[H]
\caption{}
\centering
\begin{tabular}{|c|c|c|c|c|}
\hline 
$(i,j)$ & $\alpha_{i}$ & $\alpha_{j}$ & $D(\alpha_{i} \alpha_{j})$ & $D(\alpha_{i}) \tau(\alpha_{j}) + \sigma(\alpha_{i}) D(\alpha_{j})$ \\ 
\hline \hline
$(1, 1)$ & $1$ & $1$ & $0$ & $0$ \\ 
\hline 
$(1, 2)$ & $1$ & $\sqrt{m}$ & $D(\sqrt{m})$ & $D(\sqrt{m})$ \\ 
\hline 
$(1, 3)$ & $1$ & $\sqrt{n}$ & $D(\sqrt{n})$ & $D(\sqrt{n})$ \\ 
\hline 
$(1, 4)$ & $1$ & $\sqrt{mn}$ & $D(\sqrt{mn})$ & $D(\sqrt{mn})$ \\ 
\hline 
$(2, 1)$ & $\sqrt{m}$ & $1$ & $D(\sqrt{m})$ & $D(\sqrt{m})$ \\ 
\hline 
$(2, 2)$ & $\sqrt{m}$ & $\sqrt{m}$ & $0$ & $(\sigma(\sqrt{m}) + \tau(\sqrt{m}))D(\sqrt{m})$ \\ 
\hline 
$(2, 3)$ & $\sqrt{m}$ & $\sqrt{n}$ & $D(\sqrt{mn})$ & $D(\sqrt{m}) \tau(\sqrt{n}) + \sigma(\sqrt{m}) D(\sqrt{n})$ \\ 
\hline 
$(2, 4)$ & $\sqrt{m}$ & $\sqrt{mn}$ & $mD(\sqrt{n})$ & $D(\sqrt{m}) \tau(\sqrt{mn}) + \sigma(\sqrt{m}) D(\sqrt{mn})$ \\ 
\hline 
$(3, 1)$ & $\sqrt{n}$ & $1$ & $D(\sqrt{n})$ & $D(\sqrt{n})$ \\ 
\hline 
$(3, 2)$ & $\sqrt{n}$ & $\sqrt{m}$ & $D(\sqrt{mn})$ & $D(\sqrt{n}) \tau(\sqrt{m}) + \sigma(\sqrt{n}) D(\sqrt{m})$ \\ 
\hline 
$(3, 3)$ & $\sqrt{n}$ & $\sqrt{n}$ & $0$ & $(\sigma(\sqrt{n}) + \tau(\sqrt{n}))D(\sqrt{n})$ \\ 
\hline 
$(3, 4)$ & $\sqrt{n}$ & $\sqrt{mn}$ & $n D(\sqrt{m})$ & $D(\sqrt{n}) \tau(\sqrt{mn}) + \sigma(\sqrt{n}) D(\sqrt{mn})$ \\ 
\hline 
$(4, 1)$ & $\sqrt{mn}$ & $1$ & $D(\sqrt{mn})$ & $D(\sqrt{mn})$ \\ 
\hline 
$(4, 2)$ & $\sqrt{mn}$ & $\sqrt{m}$ & $mD(\sqrt{n})$ & $D(\sqrt{mn}) \tau(\sqrt{m}) + \sigma(\sqrt{mn}) D(\sqrt{m})$ \\ 
\hline 
$(4, 3)$ & $\sqrt{mn}$ & $\sqrt{n}$ & $nD(\sqrt{m})$ & $D(\sqrt{mn}) \tau(\sqrt{n}) + \sigma(\sqrt{mn}) D(\sqrt{n})$ \\ 
\hline 
$(4, 4)$ & $\sqrt{mn}$ & $\sqrt{mn}$ & $0$ & $(\sigma(\sqrt{mn}) + \tau(\sqrt{mn}))D(\sqrt{mn})$ \\ 
\hline 
\end{tabular}
\label{table 2}
\end{table}

From Table \ref{table 2} containing the values of $D(\alpha_{i} \alpha_{j})$ and $D(\alpha_{i}) \tau(\alpha_{j}) + \sigma(\alpha_{i}) D(\alpha_{j})$ ($i, j \in \{1, 2, 3, 4\}$), we see that the relation $D(\alpha_{i} \alpha_{j}) = D(\alpha_{i}) \tau(\alpha_{j}) + \sigma(\alpha_{i}) D(\alpha_{j})$ is being satisfied by the pairs $$(i,j) = (1,1), (1,2), (1,3), (1,4), (2,1), (3,1), (4,1).$$ Then using the fact that the $\mathbb{Z}$-algebra $\mathcal{A} = \mathbb{Z}[\sqrt{m}, \sqrt{n}]$ is commutative, we get that $D$ will be a $(\sigma, \tau)$-derivation if and only if it satisfies the following nine equations simultaneously. \begin{equation} \label{eq 3.4}
(\sigma(\sqrt{m}) + \tau(\sqrt{m}))D(\sqrt{m}) = 0
\end{equation}

\begin{equation} \label{eq 3.5}
D(\sqrt{m}) \tau(\sqrt{n}) + \sigma(\sqrt{m}) D(\sqrt{n}) = D(\sqrt{mn})
\end{equation}

\begin{equation} \label{eq 3.6}
D(\sqrt{m}) \tau(\sqrt{mn}) + \sigma(\sqrt{m}) D(\sqrt{mn}) = m D(\sqrt{n})
\end{equation}

\begin{equation} \label{eq 3.7}
(\sigma(\sqrt{n}) + \tau(\sqrt{n}))D(\sqrt{n}) = 0
\end{equation}

\begin{equation} \label{eq 3.8}
D(\sqrt{n}) \tau(\sqrt{mn}) + \sigma(\sqrt{n}) D(\sqrt{mn}) =  n D(\sqrt{m})
\end{equation}

\begin{equation} \label{eq 3.9}
(\sigma(\sqrt{mn}) + \tau(\sqrt{mn}))D(\sqrt{mn}) = 0
\end{equation}

\begin{equation} \label{eq 3.10}
D(\sqrt{n}) \tau(\sqrt{m}) + \sigma(\sqrt{n}) D(\sqrt{m}) = D(\sqrt{mn})
\end{equation}

\begin{equation} \label{eq 3.11}
D(\sqrt{mn}) \tau(\sqrt{m}) + \sigma(\sqrt{mn}) D(\sqrt{m}) =  m D(\sqrt{n})
\end{equation}

\begin{equation} \label{eq 3.12}
D(\sqrt{mn}) \tau(\sqrt{n}) + \sigma(\sqrt{mn}) D(\sqrt{n}) = n D(\sqrt{m})
\end{equation}

Since $\mathcal{A} = \mathbb{Z}[\sqrt{m}, \sqrt{n}]$ is an integral domain, so in view of equations (\ref{eq 3.4}), (\ref{eq 3.7}) and (\ref{eq 3.9}), there are four possible cases.\vspace{0.2cm}

\textbf{Case 1:} All of $D(\sqrt{m})$, $D(\sqrt{n})$ and $D(\sqrt{mn})$ are non-zero.

Since $\mathcal{A} = \mathbb{Z}[\sqrt{m}, \sqrt{n}]$ is an integral domain, therefore, from (\ref{eq 3.4}), (\ref{eq 3.7}) and (\ref{eq 3.9}), $$\tau(\sqrt{m}) = - \sigma(\sqrt{m}), \hspace{0.2cm} \tau(\sqrt{n}) = - \sigma(\sqrt{n}) \hspace{0.2cm} \text{and} \hspace{0.2cm} \tau(\sqrt{mn}) = - \sigma(\sqrt{mn}).$$

But from Table \ref{table 1}, we observe that there is no such pair $(\sigma, \tau)$ which satisfies the above three equations simultaneously.
Therefore, this case is not possible.\vspace{0.2cm}

\textbf{Case 2:} Exactly one of $D(\sqrt{m})$, $D(\sqrt{n})$ and $D(\sqrt{mn})$ is non-zero.

Then there are three possible subcases.

\textbf{Subcase 2.1:} $D(\sqrt{m}) = 0$, $D(\sqrt{n}) \neq 0$, $D(\sqrt{mn}) \neq 0$.

By equations (\ref{eq 3.7}) and (\ref{eq 3.9}) respectively, $\tau(\sqrt{n}) = - \sigma(\sqrt{n})$ and $\tau(\sqrt{mn}) = - \sigma(\sqrt{mn})$. Also, by  (\ref{eq 3.5}) and (\ref{eq 3.10}), $\tau(\sqrt{m}) = \sigma(\sqrt{m})$.
Therefore, in view of Table \ref{table 1}, the possible values of the pair $(\sigma, \tau)$ are $$(\phi_{1}, \phi_{2}), \hspace{0.2cm} (\phi_{2}, \phi_{1}), \hspace{0.2cm} (\phi_{3}, \phi_{4}), \hspace{0.2cm} (\phi_{4}, \phi_{3}).$$

Further, from (\ref{eq 3.5}), $D(\sqrt{mn}) = \sigma(\sqrt{m}) D(\sqrt{n})$ and $\sigma(\sqrt{m}) = \pm \sqrt{m}$, therefore, $D$ must satisfy $D(\sqrt{mn}) = \pm \sqrt{m} D(\sqrt{n})$. More precisely, we have the following table \ref{table 3}:

\begin{table}[H]
\caption{}
\centering
\begin{tabular}{|c|c|c|c|}
\hline 
$(\sigma, \tau)$ & $D(\sqrt{mn}) = \pm \sqrt{m} D(\sqrt{n})?$ \\
\hline \hline
$(\phi_{1}, \phi_{2})$ & $D(\sqrt{mn}) = \sqrt{m} D(\sqrt{n})$ \\ 
\hline 
$(\phi_{2}, \phi_{1})$ & $D(\sqrt{mn}) = \sqrt{m} D(\sqrt{n})$ \\ 
\hline 
$(\phi_{3}, \phi_{4})$ & $D(\sqrt{mn}) = - \sqrt{m} D(\sqrt{n})$ \\ 
\hline 
$(\phi_{4}, \phi_{3})$ & $D(\sqrt{mn}) = - \sqrt{m} D(\sqrt{n})$ \\ 
\hline 
\end{tabular}
\label{table 3}
\end{table}

\textbf{Subcase 2.2:} $D(\sqrt{m}) \neq 0$, $D(\sqrt{n}) = 0$, $D(\sqrt{mn}) \neq 0$.

By equations (\ref{eq 3.4}) and (\ref{eq 3.9}) respectively, $\tau(\sqrt{m}) = - \sigma(\sqrt{m})$ and $\tau(\sqrt{mn}) = - \sigma(\sqrt{mn})$. Also, by (\ref{eq 3.5}) and (\ref{eq 3.10}), $\tau(\sqrt{n}) = \sigma(\sqrt{n})$.
Therefore, in view of Table \ref{table 1}, the possible values of the pair $(\sigma, \tau)$ are $$(\phi_{1}, \phi_{3}), \hspace{0.2cm} (\phi_{2}, \phi_{4}), \hspace{0.2cm} (\phi_{3}, \phi_{1}), \hspace{0.2cm} (\phi_{4}, \phi_{2}).$$

Further from (\ref{eq 3.10}), $D(\sqrt{mn}) = \sigma(\sqrt{n}) D(\sqrt{m})$ and $\sigma(\sqrt{n}) = \pm \sqrt{n}$, therefore, $D$ must satisfy $D(\sqrt{mn}) = \pm \sqrt{n} D(\sqrt{m})$. More precisely, we have the following table \ref{table 4}:

\begin{table}[H]
\caption{}
\centering
\begin{tabular}{|c|c|c|c|}
\hline 
$(\sigma, \tau)$ & $D(\sqrt{mn}) = \pm \sqrt{n} D(\sqrt{m})?$ \\
\hline \hline 
$(\phi_{1}, \phi_{3})$ & $D(\sqrt{mn}) = \sqrt{n} D(\sqrt{m})$ \\ 
\hline 
$(\phi_{2}, \phi_{4})$ & $D(\sqrt{mn}) = - \sqrt{n} D(\sqrt{m})$ \\ 
\hline 
$(\phi_{3}, \phi_{1})$ & $D(\sqrt{mn}) = \sqrt{n} D(\sqrt{m})$ \\ 
\hline 
$(\phi_{4}, \phi_{2})$ & $D(\sqrt{mn}) = - \sqrt{n} D(\sqrt{m})$ \\ 
\hline 
\end{tabular}
\label{table 4}
\end{table}

\textbf{Subcase 2.3:} $D(\sqrt{m}) \neq 0$, $D(\sqrt{n}) \neq 0$, $D(\sqrt{mn}) = 0$.

By equations (\ref{eq 3.4}) and (\ref{eq 3.7}) respectively, $\tau(\sqrt{m}) = - \sigma(\sqrt{m})$ and $\tau(\sqrt{n}) = - \sigma(\sqrt{n})$. Also, by equations (\ref{eq 3.6}) and (\ref{eq 3.11}), $\tau(\sqrt{mn}) = \sigma(\sqrt{mn})$. Therefore, in view of Table \ref{table 1}, the possible values of the pair $(\sigma, \tau)$ are $$(\phi_{1}, \phi_{4}), \hspace{0.2cm} (\phi_{2}, \phi_{3}), \hspace{0.2cm} (\phi_{4}, \phi_{1}), \hspace{0.2cm} (\phi_{3}, \phi_{2}).$$

Further from equation (\ref{eq 3.11}), $m D(\sqrt{n}) = \sigma(\sqrt{mn}) D(\sqrt{m})$ and $\sigma(\sqrt{mn}) = \pm \sqrt{mn}$, therefore, $D$ must satisfy $m D(\sqrt{n}) = \pm \sqrt{mn} D(\sqrt{m})$ so that $\sqrt{m} D(\sqrt{n}) = \pm \sqrt{n} D(\sqrt{m})$. More precisely, we have the following table \ref{table 5}:

\begin{table}[H]
\caption{}
\centering
\begin{tabular}{|c|c|c|c|}
\hline 
$(\sigma, \tau)$ & $\sqrt{m} D(\sqrt{n}) = \sqrt{n} D(\sqrt{m})?$ \\
\hline \hline
$(\phi_{1}, \phi_{4})$ & $\sqrt{m} D(\sqrt{n}) = \sqrt{n} D(\sqrt{m})$ \\ 
\hline 
$(\phi_{2}, \phi_{3})$ & $\sqrt{m} D(\sqrt{n}) = - \sqrt{n} D(\sqrt{m})$ \\ 
\hline 
$(\phi_{3}, \phi_{2})$ & $\sqrt{m} D(\sqrt{n}) = - \sqrt{n} D(\sqrt{m})$ \\ 
\hline 
$(\phi_{4}, \phi_{1})$ & $\sqrt{m} D(\sqrt{n}) = \sqrt{n} D(\sqrt{m})$ \\ 
\hline 
\end{tabular}
\label{table 5}
\end{table}

\textbf{Case 3:} Exactly two of $D(\sqrt{m})$, $D(\sqrt{n})$ or $D(\sqrt{mn})$ are $0$.

Then there are three possible subcases.

\textbf{Subcase 3.1:} $D(\sqrt{m}) = 0$, $D(\sqrt{n}) = 0$, $D(\sqrt{mn}) \neq 0$.

\textbf{Subcase 3.2:} $D(\sqrt{m}) = 0$, $D(\sqrt{n}) \neq 0$, $D(\sqrt{mn}) = 0$.

\textbf{Subcase 3.3:} $D(\sqrt{m}) \neq 0$, $D(\sqrt{n}) = 0$, $D(\sqrt{mn}) = 0$.

Then in each of the three possible subcases, $D$ will become a zero derivation using equations (\ref{eq 3.4}) to (\ref{eq 3.12}) and the facts that $\{1, \sqrt{m}, \sqrt{n}, \sqrt{mn}\}$ is a $\mathbb{Z}$-basis of the $\mathbb{Z}$-module $\mathcal{A} = \mathbb{Z}[\sqrt{m}, \sqrt{n}]$ and $D:\mathcal{A} \rightarrow \mathcal{A}$ is $\mathbb{Z}$-linear with $D(1)=0$. But this gives a contradiction because $D$ is non-zero. Therefore, this case is not possible.\vspace{0.2cm}

\textbf{Case 4:} All $D(\sqrt{p})$, $D(\sqrt{q})$ and $D(\sqrt{pq})$ are zero.

The proof is similar to Case 3.
\end{proof}

The proof of the following lemma is trivial using \th\ref{lemma 2.1}.
\begin{lemma}\th\label{lemma 3.17}
Let $D:\mathcal{A} \rightarrow \mathcal{A}$ be any $\mathbb{Z}$-linear map with $D(1)=0$. Then $D$ is a $(\sigma, \tau)$-derivation in the following cases.
\begin{enumerate}
\item[(i)] $D(\sqrt{m}) = 0$, $D(\sqrt{n}) \neq 0$ and $D(\sqrt{mn}) \neq 0$.
\begin{itemize}
\item[(a)] If $D(\sqrt{mn}) = \sqrt{m} D(\sqrt{n})$ and $(\sigma, \tau) =  (\phi_{1}, \phi_{2})$ or $(\phi_{2}, \phi_{1})$.

\item[(b)] If $D(\sqrt{mn}) = - \sqrt{m} D(\sqrt{n})$ and $(\sigma, \tau) =  (\phi_{3}, \phi_{4})$ or $(\phi_{4}, \phi_{3})$.
\end{itemize}

\item[(ii)] $D(\sqrt{m}) \neq 0$, $D(\sqrt{n}) = 0$ and $D(\sqrt{mn}) \neq 0$.
\begin{itemize}
\item[(a)] If $D(\sqrt{mn}) = \sqrt{n} D(\sqrt{m})$ and $(\sigma, \tau) =  (\phi_{1}, \phi_{3})$ or $(\phi_{3}, \phi_{1})$.

\item[(b)] If $D(\sqrt{mn}) = - \sqrt{n} D(\sqrt{m})$ and $(\sigma, \tau) =  (\phi_{2}, \phi_{4})$ or $(\phi_{4}, \phi_{2})$.
\end{itemize}

\item[(iii)] $D(\sqrt{m}) \neq 0$, $D(\sqrt{n}) \neq 0$ and $D(\sqrt{mn}) = 0$.
\begin{itemize}
\item[(a)] If $\sqrt{m} D(\sqrt{n}) = \sqrt{n} D(\sqrt{m})$ and $(\sigma, \tau) =  (\phi_{1}, \phi_{4})$ or $(\phi_{4}, \phi_{1})$.

\item[(b)] If $\sqrt{m} D(\sqrt{n}) = - \sqrt{n} D(\sqrt{m})$ and $(\sigma, \tau) =  (\phi_{2}, \phi_{3})$ or $(\phi_{3}, \phi_{2})$.
\end{itemize}
\end{enumerate}
\end{lemma}

\begin{theorem}\th\label{theorem 3.18}
A non-zero $\mathbb{Z}$-linear map $D:\mathcal{A} \rightarrow \mathcal{A}$ with $D(1)=0$ is a $(\sigma, \tau)$-derivation if and only if exactly one of the cases (i), (ii) or (iii) of \th\ref{lemma 3.16} holds.
\end{theorem}

\begin{corollary}\th\label{corollary 3.19}
The $\mathbb{Z}$-module $\mathcal{D}_{(\sigma, \tau)}(\mathcal{A})$ is finitely generated of rank $4$.
\end{corollary}

\begin{proof}
$\mathcal{A} = \mathbb{Z}[\sqrt{m}, \sqrt{n}]$ has exactly four distinct ring endomorphisms as given in Table \ref{table 1}. In view of \th\ref{theorem 3.18}, there are three possible cases.

\textbf{Case 1:} $(\sigma, \tau) \in \{(\phi_{1}, \phi_{2}), (\phi_{2}, \phi_{1}), (\phi_{3}, \phi_{4}), (\phi_{4}, \phi_{3})\}$.

For each $i \in \{1, 2, 3, 4\}$, define $D_{i}:\mathcal{A} \rightarrow \mathcal{A}$ as a $\mathbb{Z}$-linear map with $D(1) = 0$ such that $$D_{1}(\sqrt{n}) = 1, \hspace{0.2cm} D_{2}(\sqrt{n}) = \sqrt{m}, \hspace{0.2cm} D_{3}(\sqrt{n}) = \sqrt{n}, \hspace{0.2cm} D_{4}(\sqrt{n}) = \sqrt{mn}.$$ More precisely, for each $i \in \{1, 2, 3, 4\}$, define $D_{i}:\mathcal{A} \rightarrow \mathcal{A}$ by $$D_{i}(a_{1} + a_{2} \sqrt{m} + a_{3} \sqrt{n} + a_{4} \sqrt{mn}) = \begin{cases}
(a_{3} + a_{4} \sqrt{m}) D_{i}(\sqrt{n}) & \text{if $(\sigma, \tau) = (\phi_{1}, \phi_{2}) \hspace{0.2cm} \text{or} \hspace{0.2cm} (\phi_{2}, \phi_{1})$} \\
(a_{3} - a_{4} \sqrt{m}) D_{i}(\sqrt{n}) & \text{if $(\sigma, \tau) = (\phi_{3}, \phi_{4}) \hspace{0.2cm} \text{or} \hspace{0.2cm} (\phi_{4}, \phi_{3})$}
\end{cases}$$ where for each $i \in \{1, 2, 3, 4\}$, $D_{i}(\sqrt{n})$ is defined as above.\vspace{10pt}

\textbf{Case 2:} $(\sigma, \tau) \in \{(\phi_{1}, \phi_{3}), (\phi_{3}, \phi_{1}), (\phi_{2}, \phi_{4}), (\phi_{4}, \phi_{2})\}$.

For each $i \in \{1, 2, 3, 4\}$, define $D_{i}:\mathcal{A} \rightarrow \mathcal{A}$ as a $\mathbb{Z}$-linear map with $D(1) = 0$ such that $$D_{1}(\sqrt{m}) = 1, \hspace{0.2cm} D_{2}(\sqrt{m}) = \sqrt{m}, \hspace{0.2cm} D_{3}(\sqrt{m}) = \sqrt{n}, \hspace{0.2cm} D_{4}(\sqrt{m}) = \sqrt{mn}.$$ More precisely, for each $i \in \{1, 2, 3, 4\}$, define $D_{i}:\mathcal{A} \rightarrow \mathcal{A}$ by $$D_{i}(a_{1} + a_{2} \sqrt{m} + a_{3} \sqrt{n} + a_{4} \sqrt{mn}) = \begin{cases}
(a_{2} + a_{4} \sqrt{n}) D_{i}(\sqrt{m}) & \text{if $(\sigma, \tau) = (\phi_{1}, \phi_{3}) \hspace{0.2cm} \text{or} \hspace{0.2cm} (\phi_{3}, \phi_{1})$} \\
(a_{2} - a_{4} \sqrt{n}) D_{i}(\sqrt{m}) & \text{if $(\sigma, \tau) = (\phi_{2}, \phi_{4}) \hspace{0.2cm} \text{or} \hspace{0.2cm} (\phi_{4}, \phi_{2})$}
\end{cases}$$ where for each $i \in \{1, 2, 3, 4\}$, $D_{i}(\sqrt{m})$ is defined as above.\vspace{10pt}

\textbf{Case 3:} $(\sigma, \tau) \in \{(\phi_{1}, \phi_{4}), (\phi_{4}, \phi_{1}), (\phi_{2}, \phi_{3}), (\phi_{3}, \phi_{2})\}$.

Suppose $\text{gcd}(m, n) = k$. Then $m = kr$ and $n = ks$ for some $r, s \in \mathbb{Z}$.

For each $i \in \{1, 2, 3, 4\}$, define $D_{i}:\mathcal{A} \rightarrow \mathcal{A}$ as a $\mathbb{Z}$-linear map with $D(1) = 0$ such that $$D_{1}(\sqrt{m}) = m, \hspace{0.2cm} D_{2}(\sqrt{m}) = \sqrt{m}, \hspace{0.2cm} D_{3}(\sqrt{m}) = r \sqrt{n}, \hspace{0.2cm} D_{4}(\sqrt{m}) = \sqrt{mn}$$ so that $$D_{1}(\sqrt{n}) = \sqrt{mn}, \hspace{0.2cm} D_{2}(\sqrt{n}) = \sqrt{n}, \hspace{0.2cm} D_{3}(\sqrt{n}) = s \sqrt{m}, \hspace{0.2cm} D_{4}(\sqrt{n}) = n.$$ More precisely, for each $i \in \{1, 2, 3, 4\}$, define $D_{i}:\mathcal{A} \rightarrow \mathcal{A}$ by $$D_{i}(a_{1} + a_{2} \sqrt{m} + a_{3} \sqrt{n} + a_{4} \sqrt{mn}) = 
a_{2} D_{i}(\sqrt{m}) + a_{3} D_{i}(\sqrt{n})$$ where for each $i \in \{1, 2, 3, 4\}$, $D_{i}(\sqrt{m})$ and $D_{i}(\sqrt{n})$ are defined as above.

In all three cases, it can be easily verified that $\{D_{1}, D_{2}, D_{3}, D_{4}\}$ is a linearly independent subset of the $\mathbb{Z}$-module $\mathcal{D}_{(\sigma, \tau)}(\mathcal{A})$ that generates it.
\end{proof}

We note from the preceding corollary that the rank of the $\mathbb{Z}$-module $\mathcal{D}_{(\sigma, \tau)}(\mathcal{A})$ is equal to $4$, the degree of $K = \mathbb{Q}(\sqrt{m}, \sqrt{n})$ and which, by \th\ref{theorem 1.2}, is also equal to the rank of the free abelian group $\mathcal{A} = \mathbb{Z}[\sqrt{m}, \sqrt{n}]$ as a $\mathbb{Z}$-module.

\begin{corollary}\th\label{corollary 3.20}
There always exists a non-zero $(\sigma, \tau)$-derivation of $\mathcal{A}$.
\end{corollary}

\begin{corollary}\th\label{corollary 3.21}
There always exists a non-zero $\mathbb{Z}$-linear map $D:\mathcal{A} \rightarrow \mathcal{A}$ with $D(1)=0$ that is not a $(\sigma, \tau)$-derivation. In fact, there exists a non-zero $\mathbb{Z}$-linear map $D:\mathcal{A} \rightarrow \mathcal{A}$ with $D(1)=0$ which is not a $(\sigma, \tau)$-derivation for every pair $(\sigma, \tau)$ of any two different non-zero ring endomorphisms of $\mathcal{A}$.
\end{corollary}

\begin{theorem}\th\label{theorem 3.22}
Let $D:\mathcal{A} \rightarrow \mathcal{A}$ be a non-zero $(\sigma, \tau)$-derivation. Then $D$ is inner if and only if either $\frac{D(\sqrt{m})}{2 \sqrt{m}} \in \mathcal{A} = \mathbb{Z}[\sqrt{m}, \sqrt{n}]$ or $\frac{D(\sqrt{n})}{2 \sqrt{n}} \in \mathcal{A} = \mathbb{Z}[\sqrt{m}, \sqrt{n}]$.
\end{theorem}

\begin{proof}
$\mathcal{A} = \mathbb{Z}[\sqrt{m}, \sqrt{n}]$ has exactly four distinct ring endomorphisms as given in Table \ref{table 1}. In view of \th\ref{theorem 3.18}, there are three possible cases.

\textbf{Case 1:} $D(\sqrt{m}) = 0$, $D(\sqrt{n}) \neq 0$ and $D(\sqrt{mn}) \neq 0$.

In this case, $(\sigma, \tau) \in \{(\phi_{1}, \phi_{2}), (\phi_{2}, \phi_{1}), (\phi_{3}, \phi_{4}), (\phi_{4}, \phi_{3})\}$. In view of Subcase 2.1 in the proof of \th\ref{lemma 3.16}, in this case, $$\tau(\sqrt{m}) = \sigma(\sqrt{m}), \hspace{0.2cm} \tau(\sqrt{n}) = - \sigma(\sqrt{n}) \hspace{0.2cm} \text{and} \hspace{0.2cm} \tau(\sqrt{mn}) = - \sigma(\sqrt{mn}).$$ Note that $(\tau - \sigma)(1) = 0$, $(\tau - \sigma)(\sqrt{m}) = 0$, $(\tau - \sigma)(\sqrt{n}) = 2 \tau(\sqrt{n})$ and $(\tau - \sigma)(\sqrt{mn}) = 2 \tau(\sqrt{mn}) = 2 \tau(\sqrt{m}) \tau(\sqrt{n})$. By \th\ref{theorem 2.5}, $D$ is inner if and only if there is a $\beta \in \mathcal{A}$ with $$D(1) = \beta (\tau - \sigma)(1), \hspace{0.2cm} D(\sqrt{m}) = \beta (\tau - \sigma)(\sqrt{m}),$$ $$D(\sqrt{n}) = \beta (\tau - \sigma)(\sqrt{n}), \hspace{0.2cm} D(\sqrt{mn}) = \beta (\tau - \sigma)(\sqrt{mn}).$$

Note that $D(1) = \alpha (\tau - \sigma)(1)$ and $D(\sqrt{m}) = \alpha (\tau - \sigma)(\sqrt{m})$ hold for all $\alpha \in \mathcal{A}$. So $D$ will be inner if and only if $$D(\sqrt{n}) = \beta (\tau - \sigma)(\sqrt{n}) \hspace{0.2cm} \text{and} \hspace{0.2cm} D(\sqrt{mn}) = \beta (\tau - \sigma)(\sqrt{mn}),$$ that is, if and only if $$D(\sqrt{n}) = 2 \beta \tau (\sqrt{n}) \hspace{0.2cm} \text{and} \hspace{0.2cm} D(\sqrt{mn}) = 2 \beta \tau (\sqrt{m}) \tau (\sqrt{n}).$$

\noindent Further since in this case, $D(\sqrt{mn}) = \begin{cases} \sqrt{m} D(\sqrt{n}) \hspace{0.2cm} \text{if $\tau = \phi_{1}$ or $\phi_{2}$} \\
- \sqrt{m} D(\sqrt{n}) \hspace{0.2cm} \text{if $\tau = \phi_{3}$ or $\phi_{4}$}
\end{cases}$ and \\ $2 \beta \tau(\sqrt{m}) \tau(\sqrt{n}) = \begin{cases}
2 \beta \sqrt{mn} \hspace{0.2cm} \text{if $\tau = \phi_{1}$ or $\tau = \phi_{4}$} \\
- 2 \beta \sqrt{mn} \hspace{0.2cm} \text{if $\tau = \phi_{2}$ or $\tau = \phi_{3}$}
\end{cases}$ therefore, \\ $D(\sqrt{mn}) = 2 \beta \tau (\sqrt{m}) \tau (\sqrt{n})$ if and only if $D(\sqrt{n}) = \begin{cases}
2 \beta \sqrt{n} \hspace{0.2cm} \text{if $\tau = \phi_{1}$ or $\phi_{3}$} \\
- 2 \beta \sqrt{n} \hspace{0.2cm} \text{if $\tau = \phi_{2}$ or $\phi_{4}$}
\end{cases}.$

From the above discussion, it can be easily concluded that $D$ is inner if and only if $$D(\sqrt{n}) = \begin{cases}
2 \beta \sqrt{n} \hspace{0.2cm} \text{if $\tau = \phi_{1}$ or $\phi_{3}$} \\
- 2 \beta \sqrt{n} \hspace{0.2cm} \text{if $\tau = \phi_{2}$ or $\phi_{4}$}
\end{cases}.$$
Therefore, $D$ is inner if and only if there is a $\beta \in \mathcal{A}$ with $$\beta = \begin{cases}
\frac{D(\sqrt{n})}{2 \sqrt{n}} \hspace{0.2cm} \text{if $\tau = \phi_{1}$ or $\phi_{3}$} \\
- \frac{D(\sqrt{n})}{2 \sqrt{n}} \hspace{0.2cm} \text{if $\tau = \phi_{2}$ or $\phi_{4}$}
\end{cases}.$$

Therefore, in this case, $D$ is inner if and only if $\frac{D(\sqrt{n})}{2 \sqrt{n}} \in \mathcal{A}$.\vspace{10pt}

\textbf{Case 2:} $D(\sqrt{m}) \neq 0$, $D(\sqrt{n}) = 0$ and $D(\sqrt{mn}) \neq 0$.

As in Case 1, it can be shown that $D$ is inner if and only if there is a $\beta \in \mathcal{A}$ with $$\beta = \begin{cases}
\frac{D(\sqrt{m})}{2 \sqrt{m}} \hspace{0.2cm} \text{if $\tau = \phi_{1}$ or $\phi_{2}$} \\
- \frac{D(\sqrt{m})}{2 \sqrt{m}} \hspace{0.2cm} \text{if $\tau = \phi_{3}$ or $\phi_{4}$}
\end{cases}.$$

Therefore, in this case, $D$ is inner if and only if $\frac{D(\sqrt{m})}{2 \sqrt{m}} \in \mathcal{A}$.\vspace{10pt}

\textbf{Case 3:} $D(\sqrt{m}) \neq 0$, $D(\sqrt{n}) \neq 0$ and $D(\sqrt{mn}) = 0$.

On similar lines, it can be shown that in this case, $D$ is inner if and only if either $\frac{D(\sqrt{m})}{2 \sqrt{m}} \in \mathcal{A}$ or $\frac{D(\sqrt{n})}{2 \sqrt{n}} \in \mathcal{A}$ (as $\frac{D(\sqrt{m})}{2 \sqrt{m}} = \frac{D(\sqrt{n})}{2 \sqrt{n}}$).

Hence the result.
\end{proof}

The corollaries are now immediate.

\begin{corollary}\th\label{corollary 3.23}
$\mathcal{A}$ has non-trivial inner $(\sigma, \tau)$-derivations, that is, $\text{Inn}_{(\sigma, \tau)}(\mathcal{A}) \neq \{0\}$.
\end{corollary}

\begin{corollary}
$\mathcal{A}$ has non-trivial outer $(\sigma, \tau)$-derivations, that is, $\text{Out}_{(\sigma, \tau)}(\mathcal{A}) \neq \{0\}$.
\end{corollary}

\section{Coding Theory Applications}\label{section 4}
In this section, we give the applications of the work of the previous sections in coding theory. Here, we give the notion of a Hom-IDD code over finite fields. Initially, the notion of IDD codes for derivations of group rings was given in \cite{Manju2024}. Here we first give the equivalent notion of an IDD code over $\mathbb{Z}$, and then give the notion of Hom-IDD code over a finite ring of characteristic $q$. We illustrate our construction of such codes with examples.

\subsection{Construction and Definition}\label{subsection 4.1}
Let $K = \mathbb{Q}(\theta)$ ($\theta$ an algebraic integer) be a number field of degree $n$ over $\mathbb{Q}$, which is a normal extension of $\mathbb{Q}$ and $\mathcal{B} = \{\theta_{1} = \theta, \theta_{2}, ..., \theta_{n}\}$ be an ordered integral basis of $O_{K}$. Let $\sigma, \tau:O_{K} \rightarrow O_{K}$ be two different $\mathbb{Z}$-algebra endomorphisms of $O_{K}$. Let $D$ be a $(\sigma, \tau)$-derivation of $O_{K}$.

The range $\text{Im}(D)$ of $D$ is a $\mathbb{Z}$-submodule of the $\mathbb{Z}$-module $O_{K}$ with the generating set $D(\mathcal{B}) = \{D(\theta_{i}) | 1 \leq i \leq n\}$. It generates an $(n,r)$-code in $O_{K}$, where $r = \text{rank}(\text{Im}(D))$, the rank of the $\mathbb{Z}$-module $\text{Im}(D)$.

Suppose $S = \{D(\theta_{i_{1}}),D(\theta_{i_{2}}),...,D(\theta_{i_{s}})\}$ is a $\mathbb{Z}$-linearly independent subset of $\text{Im}(D)$. Then $S$ generates a submodule of $\text{Im}(D)$ with rank $s$. Thus, S generates an $(n,s)$-code.

\begin{definition}
Let $T = \{\theta_{k_{1}}, \theta_{k_{2}}, ..., \theta_{k_{t}}\}$ ($\{k_{1}, k_{2}, ..., k_{t}\} \subseteq \{1, 2, ..., n\}$) be a subset of the integral basis $\{\theta_{1}, \theta_{2}, ..., \theta_{n}\}$ such that $D(T)$ is $\mathbb{Z}$-linearly independent subset of $O_{K}$ (or $\text{Im}(D)$). Then the submodule $W = \langle D(T) \rangle$ generates an $(n,t)$-code, and is called the Image of Derivation-Derived Code, or simply an IDD Code.
\end{definition}

\subsection{Equivalent IDD Code in $\mathbb{Z}^{n}$}\label{subsection 4.2}
In this subsection, we construct a code in $\mathbb{Z}^{n}$ equivalent to an IDD code.

Suppose $\Delta:\mathbb{Z}^{n} \rightarrow O_{K}$ be the mapping given by $$\Delta(a_{1}, a_{2}, ..., a_{n}) = \sum_{i=1}^{n} a_{i} \theta_{i}.$$

For each $i \in \{1, 2, ..., n\}$, suppose $D(\theta_{i}) = \sum_{j=1}^{n} a_{ij}\theta_{j}$. Then $\Delta^{-1}(D(\theta_{i})) = (a_{i1}, a_{i2}, ..., a_{in})$.

Construct an $n \times n$ matrix $B$ with the rows $\Delta^{-1}(D(\theta_{1})), \Delta^{-1}(D(\theta_{2})), ..., \Delta^{-1}(D(\theta_{n}))$. More precisely, $$B = \begin{pmatrix}
a_{11} & a_{12} & \cdots & a_{1n} \\
a_{21} & a_{22} & \cdots & a_{2n} \\
\vdots & \vdots & \ddots & \vdots \\
a_{n1} & a_{n2} & ... & a_{nn} \\
\end{pmatrix}.$$ 

\begin{theorem}
Any rows $i_{1}, i_{2}, ..., i_{t}$ of $B$ are linearly independent if and only if the set $\{D(\theta_{i_{1}}), D(\theta_{i_{2}}), ..., D(\theta_{i_{t}})\}$ is linearly independent.
\end{theorem}

Below, we give the steps of construction:
\begin{enumerate}
\item[(i)] Suppose $W$ is an Image of Derivations-Derived (IDD) code. Then $W = \langle S \rangle$, for some $\mathbb{Z}$-linearly independent subset $S$ of $D(\mathcal{B})$, say, $S = \{D(\theta_{k_{1}}), D(\theta_{k_{2}}), ..., D(\theta_{k_{s}})\}$.\vspace{6pt}

\item[(ii)] Pick $\underline{w} \in \mathbb{Z}^{s}$ and $\underline{w} = (\lambda_{1}, \lambda_{2}, ..., \lambda_{s})$.\vspace{6pt}

\item[(iii)] Write $\underline{w}$ as $\underline{x}$ in $\mathbb{Z}^{n}$ with $\lambda_{j}$ in position $k_{j}$ and zero elsewhere.\vspace{6pt}

\item[(iv)] $\underline{x}$ can be mapped to an element $O_{K}$ by $x = \Delta(\underline{x}) = \sum_{j=1}^{s}\lambda_{j}\theta_{k_{j}}$ and a codeword $$D(x) = D(\Delta(\underline{x})) = D(\sum_{j=1}^{s}\lambda_{j}\theta_{k_{j}}) = \sum_{j=1}^{s}\lambda_{j}D(\theta_{k_{j}})$$ equated with a codeword in $\mathcal{E}$ given by $$\Delta^{-1}(D(x)) = \underline{x}B.$$

\item[(v)] Thus, we obtain an $(n,s)$ code $$\mathcal{E} = \{\underline{x}B \mid \underline{x} ~ \text{extends} ~ \underline{w} \in \mathbb{Z}^{s} ~ \text{to} ~ \mathbb{Z}^{n}\}.$$
\end{enumerate}

\subsection{Hom-IDD Code in $(\mathbb{Z}_{q})^{n}$}\label{subsection 4.3}
For a positive integer $q$ and for an equivalent IDD code $C$ in $\mathbb{Z}^{n}$, we construct a code $\Omega(C)$ in $(\mathbb{Z}_{q})^{n}$. We call this newly constructed code $\Omega(C)$ the Hom-IDD code.

\begin{lemma}\label{lemma 4.3}
Let $q$ be a positive integer. Then the map $\Omega:\mathbb{Z}^{n} \rightarrow (\mathbb{Z}_{q})^{n}$ defined by $$\Omega(\lambda_{1}, \lambda_{2}, ..., \lambda_{n}) = (\overline{\lambda_{1}}, \overline{\lambda_{2}}, ..., \overline{\lambda_{n}}),$$ where $\overline{\lambda_{i}} = \lambda_{i}$ (mod $q$) for each $i \in \{1, 2, ..., n\}$ is a $\mathbb{Z}$-module homomorphism.
\end{lemma}
\begin{proof}
Both $\mathbb{Z}^{n}$ and $(\mathbb{Z}_{q})^{n}$ are abelian groups with respect to the usual operations of componentwise addition. And since every abelian group is a $\mathbb{Z}$-module with respect to the usual operations of scalar multiplication, therefore, both $\mathbb{Z}^{n}$ and $(\mathbb{Z}_{q})^{n}$ are $\mathbb{Z}$-modules. More precisely, if $M$ is an additive abelian group, then $M$ is a $\mathbb{Z}$-module with respect to the scalar multiplication operation defined by: $\mathbb{Z} \times M \rightarrow M$, where $$(\lambda, m) \mapsto \lambda m,$$ where $$\lambda m = \begin{cases} 
\underbrace{m + m + ... + m}_{\lambda ~ \text{times}} & \text{if} ~ \lambda > 0 \\
\underbrace{(-m) + (-m) + ... + (-m)}_{-\lambda ~ \text{times}} & \text{if} ~ \lambda < 0 \\
0 & \text{if} ~ \lambda = 0
\end{cases}.$$

First note that $\Omega$ is indeed a well-defined map from $\mathbb{Z}^{n}$ to $(\mathbb{Z}_{q})^{n}$.

Let $x = (\lambda_{1}, \lambda_{2}, ..., \lambda_{n}), y = (\mu_{1}, \mu_{2}, ..., \mu_{n}) \in \mathbb{Z}^{n}$ and $\lambda \in \mathbb{Z}$.

Then \begin{align*}
\Omega(x +_{1} y) & = \Omega(\lambda_{1} + \mu_{1}, \lambda_{2} + \mu_{2}, ..., \lambda_{n} + \mu_{n}) \\
& = (\overline{\lambda_{1} + \mu_{1}}, \overline{\lambda_{2} + \mu_{2}}, ..., \overline{\lambda_{n} + \mu_{n}}),
\end{align*} where for each $i \in \{1, 2, ..., n\}$, $\overline{\lambda_{i} + \mu_{i}} = (\lambda_{i} + \mu_{i}) (\text{mod} ~ q)$. Here, $+_{1}$ denotes the addition operation in $\mathbb{Z}^{n}$ and $+$ denotes the addition operation in $\mathbb{Z}$.

Now, for $i \in \{1, 2, ..., n\}$, we have that \begin{align*}
\overline{\lambda_{i} + \mu_{i}} & = (\lambda_{i} + \mu_{i}) (\text{mod} ~ q) \\
& = (\lambda_{i} (\text{mod} ~ q) + \mu_{i} (\text{mod} ~ n))(\text{mod} ~ q) \\
& = (\overline{\lambda_{i}} + \overline{\mu_{i}})(\text{mod} ~ q) \\
& = \overline{\lambda_{i}} \oplus_{q} \overline{\mu_{i}}
\end{align*} Here, $\oplus_{q}$ denotes the addition operation in $\mathbb{Z}_{q}$.

So from above, we get that \begin{align*}
\Omega(x +_{1} y) & = (\overline{\lambda_{1} + \mu_{1}}, \overline{\lambda_{2} + \mu_{2}}, ..., \overline{\lambda_{n} + \mu_{n}}) \\
& = (\overline{\lambda_{1}} \oplus_{q} \overline{\mu_{1}}, \overline{\lambda_{2}} \oplus_{q} \overline{\mu_{2}}, ..., \overline{\lambda_{n}} \oplus_{q} \overline{\mu_{n}}) \\
& = (\overline{\lambda_{1}}, \overline{\lambda_{2}}, ..., \overline{\lambda_{n}}) +_{2} (\overline{\mu_{1}}, \overline{\mu_{2}}, ..., \overline{\mu_{n}})
\\ & = \Omega(x) +_{2} \Omega(y)
\end{align*} Here, $+_{2}$ denotes the addition operation in $(\mathbb{Z}_{q})^{n}$.

Now, first suppose that $\lambda > 0$. \begin{align*}
\lambda \odot_{1} x & = \underbrace{x +_{1} x +_{1} ... +_{1} x}_{\lambda ~ \text{times}} \\
& = (\underbrace{\lambda_{1} + \lambda_{1} + ... + \lambda_{1}}_{\lambda ~ \text{times}}, \underbrace{\lambda_{2} + \lambda_{2} + ... + \lambda_{2}}_{\lambda ~ \text{times}}, ..., \underbrace{\lambda_{n} + \lambda_{n} + ... + \lambda_{n}}_{\lambda ~ \text{times}}) \\
& = (\lambda \lambda_{1}, \lambda \lambda_{2}, ..., \lambda \lambda_{n})
\end{align*} Here, $\odot_{1}$ denotes the scalar multiplication operation in the $\mathbb{Z}$-module $\mathbb{Z}^{n}$.

Now, \begin{align*}
\Omega(\lambda \odot_{1} x) & = \Omega(\lambda \lambda_{1}, \lambda \lambda_{2}, ..., \lambda \lambda_{n}) \\
& = (\overline{\lambda  \lambda_{1}}, \overline{\lambda  \lambda_{2}}, ..., \overline{\lambda  \lambda_{n}})
\end{align*}

For each $i \in \{1, 2, ..., n\}$,
\begin{align*}
\overline{\lambda \lambda_{i}} & = (\lambda \lambda_{i})(\text{mod} ~ q) 
\end{align*}

\begin{align*}
\lambda \odot_{2} \Omega(x) & = \underbrace{\Omega(x) +_{2} \Omega(x) +_{2} ... +_{2} \Omega(x)}_{\lambda ~ \text{times}} \\
& = (\underbrace{\overline{\lambda_{1}} \oplus_{q} \overline{\lambda_{1}} \oplus_{q} ... \oplus_{q} \overline{\lambda_{1}}}_{\lambda ~ \text{times}}, \underbrace{\overline{\lambda_{2}} \oplus_{q} \overline{\lambda_{2}} \oplus_{q} ... \oplus_{q} \overline{\lambda_{2}}}_{\lambda ~ \text{times}}, ..., \underbrace{\overline{\lambda_{n}} \oplus_{q} \overline{\lambda_{n}} \oplus_{q} ... \oplus_{q} \overline{\lambda_{n}}}_{\lambda ~ \text{times}})
\end{align*} Here, $\odot_{2}$ denotes the scalar multiplication operation in the $\mathbb{Z}$-module $(\mathbb{Z}_{q})^{n}$.

Now, let $i \in \{1, 2, ..., n\}$. Then we get that \begin{align*}
\underbrace{\overline{\lambda_{i}} \oplus_{q} \overline{\lambda_{i}} \oplus_{q} ... \oplus_{q} \overline{\lambda_{i}}}_{\lambda ~ \text{times}}
& = (\underbrace{\overline{\lambda_{i}} + \overline{\lambda_{i}} + ... + \overline{\lambda_{i}}}_{\lambda ~ \text{times}}) (\text{mod} ~ q) \\
& = (\underbrace{\lambda_{i} + \lambda_{i} + ... + \lambda_{i}}_{\lambda ~ \text{times}}) (\text{mod} ~ q) \\
& = (\lambda \lambda_{i}) (\text{mod} ~ q)
\end{align*}

Above, we have used the result: \textit{Let $n$ be a fixed positive integer. If $a (\text{mod} ~ n) = a'$ and $b (\text{mod} ~ n) = b'$, then $(a + b) (\text{mod} ~ n) = (a' + b') (\text{mod} ~ n)$ and $(ab) (\text{mod} ~ n) = (a' b') (\text{mod} ~ n)$.}

So we have proved that $\Omega(\lambda \odot_{1} x) = \lambda \odot_{2} \Omega(x)$ for $\lambda > 0$.

The above equation holds trivially for $\lambda = 0$.

Now, suppose $\lambda < 0$. Then \begin{align*}
\lambda \odot_{1} x & = \underbrace{(-x) +_{1} (-x) +_{1} ... +_{1} (-x)}_{- \lambda ~ \text{times}} \\
& = (\underbrace{(-\lambda_{1}) + (-\lambda_{1}) + ... + (-\lambda_{1})}_{- \lambda ~ \text{times}}, \underbrace{(- \lambda_{2}) + (- \lambda_{2}) + ... + (-\lambda_{2})}_{- \lambda ~ \text{times}}, ..., \\ &\quad \underbrace{(- \lambda_{n}) + (- \lambda_{n}) + ... + (- \lambda_{n})}_{- \lambda ~ \text{times}}) \\
& = ((- \lambda) (- \lambda_{1}), (- \lambda) (- \lambda_{2}), ..., (- \lambda) (- \lambda_{n})) \\
& = (\lambda \lambda_{1}, \lambda \lambda_{2}, ..., \lambda \lambda_{n})
\end{align*}

Then as shown above, we will have $$\Omega(\lambda \odot_{1} x) = (\overline{\lambda  \lambda_{1}}, \overline{\lambda  \lambda_{2}}, ..., \overline{\lambda  \lambda_{n}}),$$ where for each $i \in \{1, 2, ..., n\}$, $\overline{\lambda  \lambda_{i}} = (\lambda \lambda_{i}) (\text{mod} ~ q)$.

Using the fact that $\Omega$ is a group homomorphism from the additive group $\mathbb{Z}^{n}$ to the additive group $(\mathbb{Z}_{q})^{n}$, we get that \begin{align*}
\lambda \odot_{2} \Omega(x) & = \underbrace{(- \Omega(x)) +_{2} (- \Omega(x)) +_{2} ... +_{2} (- \Omega(x))}_{- \lambda ~ \text{times}} \\
& = \underbrace{\Omega(-x) +_{2} \Omega(-x) +_{2} ... +_{2}  \Omega(-x)}_{- \lambda ~ \text{times}} \\
& = (\underbrace{\overline{- \lambda_{1}} \oplus_{q} \overline{- \lambda_{1}} \oplus_{q} ... \oplus_{q} \overline{- \lambda_{1}}}_{- \lambda ~ \text{times}}, \underbrace{\overline{- \lambda_{2}} \oplus_{q} \overline{- \lambda_{2}} \oplus_{q} ... \oplus_{q} \overline{- \lambda_{2}}}_{- \lambda ~ \text{times}}, ..., \\ &\quad \underbrace{\overline{- \lambda_{n}} \oplus_{q} \overline{- \lambda_{n}} \oplus_{q} ... \oplus_{q} \overline{- \lambda_{n}}}_{- \lambda ~ \text{times}})
\end{align*}

Now, for $i \in \{1, 2, ..., n\}$, \begin{align*}
\underbrace{\overline{- \lambda_{i}} \oplus_{q} \overline{- \lambda_{i}} \oplus_{q} ... \oplus_{q} \overline{- \lambda_{i}}}_{- \lambda ~ \text{times}}
& = (\underbrace{(\overline{- \lambda_{i}}) + (\overline{- \lambda_{i}}) + ... + (\overline{- \lambda_{i}})}_{- \lambda ~ \text{times}}) (\text{mod} ~ q) \\
& = (\underbrace{(- \lambda_{i}) + (- \lambda_{i}) + ... + (- \lambda_{i})}_{- \lambda ~ \text{times}}) (\text{mod} ~ q) \\
& = ((-\lambda) (- \lambda_{i}) (\text{mod} ~ q) \\
& = (\lambda \lambda_{i}) (\text{mod} ~ q)
\end{align*}

So we have proved that $\Omega(\lambda \odot_{1} x) = \lambda \odot_{2} \Omega(x)$ for $\lambda < 0$.

In all possible cases, we have proved that $\Omega(\lambda \odot_{1} x) = \lambda \odot_{2} \Omega(x)$ for all $\lambda \in \mathbb{Z}$.

Therefore, $\Omega$ is a $\mathbb{Z}$-module homomorphism from the $\mathbb{Z}$-module $\mathbb{Z}^{n}$ to the $\mathbb{Z}$-module $(\mathbb{Z}_{q})^{n}$.

Hence proved.
\end{proof}

\begin{lemma}\label{lemma 4.4}
\begin{enumerate}
\item[(i)] If $W$ is a $\mathbb{Z}$-submodule of $\mathbb{Z}^{n}$, then $\Omega(W)$ is a $\mathbb{Z}$-submodule of $(\mathbb{Z}_{q})^{n}$.
\item[(ii)] If $V$ is a $\mathbb{Z}$-submodule of $(\mathbb{Z}_{q})^{n}$, then $\Omega^{-1}(V) = \{w \in \mathbb{Z}^{n} \mid \Omega(w) \in V\}$ is a $\mathbb{Z}$-submodule of $\mathbb{Z}^{n}$.
\end{enumerate}
\end{lemma}
\begin{proof}
Follows very easily from Lemma \ref{lemma 4.3}.
\end{proof}

\begin{lemma}\label{lemma 4.5}
A subset $W$ is a $\mathbb{Z}$-submodule of $(\mathbb{Z}_{q})^{n}$ if and only if $W$ is a $\mathbb{Z}_{q}$-submodule of $(\mathbb{Z}_{q})^{n}$.
\end{lemma}
\begin{proof}
Since $(\mathbb{Z}_{q})^{n}$ is a commutative unital ring with respect to the usual operations of componentwise addition module $q$ and multiplication module $q$, therefore, $(\mathbb{Z}_{q})^{n}$ is also a $\mathbb{Z}_{q}$-module with respect to the usual operation of componentwise scalar multiplication modulo $q$.

Let $\odot_{2}$ denote the scalar multiplication of the $\mathbb{Z}$-module $(\mathbb{Z}_{q})^{n}$ and $\odot$ denote the scalar multiplication of the $\mathbb{Z}_{q}$-module $(\mathbb{Z}_{q})^{n}$.

First, let $W$ be a $\mathbb{Z}$-submodule of $(\mathbb{Z}_{q})^{n}$.

We show that $W$ is a $\mathbb{Z}_{q}$-submodule of $(\mathbb{Z}_{q})^{n}$.

$W$ being a $\mathbb{Z}$-submodule of $(\mathbb{Z}_{q})^{n}$ is a non-empty subset of $(\mathbb{Z}_{q})^{n}$.

Let $w, w' \in W$ and $\lambda \in \mathbb{Z}_{q}$. Suppose $w = (\lambda_{1}, \lambda_{2}, ..., \lambda_{n})$ for some $\lambda_{i} \in \mathbb{Z}_{q}$ ($1 \leq i \leq n$).

Since $w, w' \in W$ and $W$ is a $\mathbb{Z}$-submodule of $(\mathbb{Z}_{q})^{n}$, therefore, $w+w' \in W$.

$\lambda \in \mathbb{Z}_{q}$ implies that $\lambda \in \mathbb{Z}$. 

Now, \begin{align*}
\lambda \odot_{2} w & = \underbrace{w +_{2} w +_{2} ... +_{2} w}_{\lambda ~ \text{times}} \\
& = (\underbrace{\lambda_{1} \oplus_{q} \lambda_{1} \oplus_{q} ... \oplus_{q} \lambda_{1}}_{\lambda ~ \text{times}}, \underbrace{\lambda_{2} \oplus_{q} \lambda_{2} \oplus_{q} ... \oplus_{q} \lambda_{2}}_{\lambda ~ \text{times}}, ..., \underbrace{\lambda_{n} \oplus_{q} \lambda_{n} \oplus_{q} ... \oplus_{q} \lambda_{n}}_{\lambda ~ \text{times}})
\end{align*}

Also, $$\lambda \odot w = (\lambda \otimes_{q} \lambda_{1}, \lambda \otimes_{q} \lambda_{2}, ..., \lambda \otimes_{q} \lambda_{n})$$

Here, $\oplus_{q}$ denotes addition modulo $q$ and $\otimes_{q}$ denotes multiplication modulo $q$.

\textbf{Claim: For each $i \in \{1, 2, ..., n\}$, $\underbrace{\lambda_{i} \oplus_{q} \lambda_{i} \oplus_{q} ... \oplus_{q} \lambda_{i}}_{\lambda ~ \text{times}} = \lambda \otimes_{q} \lambda_{1}$.}

So let $i \in \{1, 2, ..., n\}$.

Then we get \begin{align*}
\underbrace{\lambda_{i} \oplus_{q} \lambda_{i} \oplus_{q} ... \oplus_{q} \lambda_{i}}_{\lambda ~ \text{times}} & = (\underbrace{\lambda_{i} + \lambda_{i} + ... + \lambda_{i}}_{\lambda ~ \text{times}}) (\text{mod} ~ q) \\
& = (\lambda \lambda_{i}) (\text{mod} ~ q) \\
& = \lambda \otimes_{q} \lambda_{i}
\end{align*} since $\lambda, \lambda_{i} \in \mathbb{Z}_{q}$.

So we have proved that $\lambda \odot_{2} w = \lambda \odot w$.

Since $W$ is a $\mathbb{Z}$-submodule, therefore, $\lambda \odot_{2} w \in W$.

Further, since $\lambda \odot_{2} w = \lambda \odot w$, therefore, $\lambda \odot w \in W$.

Since $w, w' \in W$ and $\lambda \in \mathbb{Z}_{q}$ are arbitrary, therefore, $w+w', \lambda \odot w \in W$ for all $w, w' \in W$ and $\lambda \in \mathbb{Z}_{q}$.

Therefore, $W$ is also a $\mathbb{Z}_{q}$-submodule of $(\mathbb{Z}_{q})^{n}$.

Conversely, let $W$ be a $\mathbb{Z}_{q}$-submodule of $(\mathbb{Z}_{q})^{n}$.

We show that $W$ is also a $\mathbb{Z}$-submodule of $(\mathbb{Z}_{q})^{n}$.

Being a $\mathbb{Z}_{q}$-submodule of $(\mathbb{Z}_{q})^{n}$, $W$ is a non-empty subset of $(\mathbb{Z}_{q})^{n}$.

Let $w, w' \in W$ and $\lambda \in \mathbb{Z}$.

Since $W$ is a $\mathbb{Z}_{q}$-submodule of $(\mathbb{Z}_{q})^{n}$, therefore, $w + w' \in W$.

Since $w \in W \subseteq (\mathbb{Z}_{q})^{n}$, so $w = (\lambda_{1}, \lambda_{2}, ..., \lambda_{n})$ for some $\lambda_{i} \in \mathbb{Z}_{q}$ ($1 \leq i \leq n$).

Now, if $\lambda > 0$, then \begin{align*}
\lambda \odot_{2} w & = \underbrace{w +_{2} w +_{2} ... +_{2} w}_{\lambda ~ \text{times}} \\
& = (\underbrace{\lambda_{1} \oplus_{q} \lambda_{1} \oplus_{q} ... \oplus_{q} \lambda_{1}}_{\lambda ~ \text{times}}, \underbrace{\lambda_{2} \oplus_{q} \lambda_{2} \oplus_{q} ... \oplus_{q} \lambda_{2}}_{\lambda ~ \text{times}}, ..., \underbrace{\lambda_{n} \oplus_{q} \lambda_{n} \oplus_{q} ... \oplus_{q} \lambda_{n}}_{\lambda ~ \text{times}}).
\end{align*} 

As $\overline{\lambda} = \lambda (\text{mod} ~ q)$, so
\begin{align*}
\overline{\lambda} \odot w & = (\overline{\lambda} \otimes_{q} \lambda_{1}, \overline{\lambda} \otimes_{q} \lambda_{2}, ..., \overline{\lambda} \otimes_{q} \lambda_{n}).
\end{align*}

\textbf{Claim: For each $i \in \{1, 2, ..., n\}$, $\underbrace{\lambda_{i} \oplus_{q} \lambda_{i} \oplus_{q} ... \oplus_{q} \lambda_{i}}_{\lambda ~ \text{times}} = \overline{\lambda} \otimes_{q} \lambda_{i}$.}

We get \begin{align*}
\underbrace{\lambda_{i} \oplus_{q} \lambda_{i} \oplus_{q} ... \oplus_{q} \lambda_{i}}_{\lambda ~ \text{times}} & = (\underbrace{\lambda_{i} + \lambda_{i} + ... + \lambda_{i}}_{\lambda ~ \text{times}}) (\text{mod} ~ q) \\
& = (\lambda \lambda_{i}) (\text{mod} ~ q) \\
& = (\overline{\lambda} \overline{\lambda_{i}}) (\text{mod} ~ q) \\
& = \overline{\lambda} \otimes_{q} \lambda_{i}
\end{align*} because $\overline{\lambda_{i}} = \lambda_{i}$ as $\lambda_{i} \in \mathbb{Z}_{q}$.

So we have proved that for $\lambda > 0$, $\lambda \odot_{2} w = \overline{\lambda} \odot w$.

For $\lambda = 0$, the result holds trivially.

If $\lambda < 0$, then \begin{align*}
\lambda \odot_{2} w & = \underbrace{(-w) +_{2} (-w) +_{2} ... +_{2} (-w)}_{- \lambda ~ \text{times}} \\
& = (\underbrace{(- \lambda_{1}) \oplus_{q} (- \lambda_{1}) \oplus_{q} ... \oplus_{q} (- \lambda_{1})}_{- \lambda ~ \text{times}}, \underbrace{(- \lambda_{2}) \oplus_{q} (- \lambda_{2}) \oplus_{q} ... \oplus_{q} (- \lambda_{2})}_{- \lambda ~ \text{times}}, ..., \\ &\quad \underbrace{(- \lambda_{n}) \oplus_{q} (- \lambda_{n}) \oplus_{q} ... \oplus_{q} (- \lambda_{n})}_{- \lambda ~ \text{times}})
\end{align*}

\textbf{Claim: For each $i \in \{1, 2, ..., n\}$, $\underbrace{(- \lambda_{i}) \oplus_{q} (- \lambda_{i}) \oplus_{q} ... \oplus_{q} (- \lambda_{i})}_{- \lambda ~ \text{times}} = \overline{\lambda} \otimes_{q} \lambda_{i}$.}

We get \begin{align*}
\underbrace{(- \lambda_{i}) \oplus_{q} (- \lambda_{i}) \oplus_{q} ... \oplus_{q} (- \lambda_{i})}_{- \lambda ~ \text{times}} & = (\underbrace{(- \lambda_{i}) + (- \lambda_{i}) + ... + (- \lambda_{i})}_{- \lambda ~ \text{times}}) (\text{mod} ~ q) 
\\ & = ((- \lambda) (- \lambda_{i})) (\text{mod} ~ q)
\\ & = (\lambda \lambda_{i}) (\text{mod} ~ q) 
\\ & = (\overline{\lambda} \overline{\lambda_{i}}) (\text{mod} ~ q)
\\ & = \overline{\lambda} \otimes_{q} \lambda_{i}
\end{align*} because $\overline{\lambda_{i}} = \lambda_{i}$ as $\lambda_{i} \in \mathbb{Z}_{q}$.

So we have proved that for $\lambda < 0$, $\lambda \odot_{2} w = \overline{\lambda} \odot w$.

Therefore, for any $\lambda \in \mathbb{Z}$, $\lambda \odot_{2} w = \overline{\lambda} \odot w$.

But since $\overline{\lambda} \in \mathbb{Z}_{q}$, $w \in W$ and $W$ is a $\mathbb{Z}_{q}$-submodule of $(\mathbb{Z}_{q})^{n}$, therefore, $\overline{\lambda} \odot w \in W$.

And then the fact that $\lambda \odot_{2} w = \overline{\lambda} \odot w$ implies that $\lambda \odot_{2} w \in W$.

Therefore, $W$ is also a $\mathbb{Z}$-submodule of $(\mathbb{Z}_{q})^{n}$.
\end{proof}

Suppose that $W$ is an IDD code and $\mathcal{E}$ is its corresponding equivalent IDD code in $\mathbb{Z}^{n}$. Then $\mathcal{E}$ is a $\mathbb{Z}$-submodule of the $\mathbb{Z}$-module $\mathbb{Z}^{n}$. By Lemma \ref{lemma 4.4} (i), $\Omega(\mathcal{E})$ is a $\mathbb{Z}$-submodule of $(\mathbb{Z}_{q})^{n}$. Then by Lemma \ref{lemma 4.5}, $\Omega(\mathcal{E})$ is a $\mathbb{Z}_{q}$-submodule of $(\mathbb{Z}_{q})^{n}$. Hence, $\Omega(\mathcal{E})$ is a code over $\mathbb{Z}_{q}$ and in $(\mathbb{Z}_{q})^{n}$. So we have the following definition.

\begin{definition}
Let $q$ be a positive integer. With the notations of Subsection \ref{subsection 4.1}, let $W$ be an IDD code and $\mathcal{E}$ be its equivalent IDD code in $\mathbb{Z}^{n}$. Then the $\mathbb{Z}_{q}$-submodule $\Omega(\mathcal{E})$ of $(\mathbb{Z}_{q})^{n}$ is called Hom-IDD code in $(\mathbb{Z}_{q})^{n}$.
\end{definition}

\subsection{Examples}\label{subsection 4.4}
\begin{example}
Let $K = \mathbb{Q}(\zeta)$, $\zeta$ primitive $p^{\text{th}}$-root of unity, where $p=13$. Let $\sigma(\zeta) = \zeta$ and $\tau(\zeta) = \zeta^{2}$. Then by Theorem \ref{theorem 3.8}, the map $D:O_{K} \rightarrow O_{K}$ given by $$D(\zeta) = 1 + \zeta + \zeta^{2} + \zeta^{3} + \zeta^{5} +  \zeta^{8}$$ is a $(\sigma, \tau)$-derivation of $O_{K} = \mathbb{Z}[\zeta]$. 

\textbf{(i)} Below, we have obtained binary Hom-IDD codes of length $n=12$ over the finite field $\mathbb{Z}_{2}$ with various parameters, using the construction discussed above.

\begin{center}
\begin{longtable}{|c|c|c|c|}
\hline 
\textbf{Basis $S_{i}$} & \thead{\textbf{Code Description:}\\ $[n,k,d]$} & \textbf{Code Properties} & \thead{\textbf{Dual Code} \\ \textbf{Description:} \\ $[n,k,d]$} \\ 
\hline \hline
$S_{1}$ & $[12,5,4]$ & \thead{non-LCD \& optimal} & $[12,7,3]$ \\
\hline
$S_{2}$ & $[12,5,3]$ & LCD & $[12,7,3]$ \\
\hline
$S_{3}$ & $[12,6,3]$ & non-LCD & $[12,6,3]$ \\
\hline
$S_{4}$ & $[12,6,3]$ & LCD & $[12,6,3]$ \\
\hline
$S_{5}$ & $[12,7,3]$ & LCD & $[12,5,3]$ \\
\hline
$S_{6}$ & $[12,5,4]$ & \thead{LCD \& optimal} & $[12,7,3]$ \\
\hline
$S_{7}$ & $[12,4,5]$ & non-LCD & \thead{$[12,8,3]$ \\ (optimal)} \\
\hline
$S_{8}$ & $[12,4,5]$ & LCD & $[12,8,2]$ \\
\hline
$S_{9}$ & $[12,4,4]$ & LCD & $[12,8,1]$ \\
\hline
$S_{10}$ & $[12,3,5]$ & LCD & $[12,9,1]$ \\
\hline
$S_{11}$ & $[12,7,3]$ & non-LCD & $[12,5,3]$ \\
\hline
$S_{12}$ & $[12,8,2]$ & LCD & $[12,4,4]$ \\
\hline
\end{longtable}
\end{center}

\begin{center}
\begin{longtable}{|c|}
\hline 
\textbf{Basis $S_{i}$} \\ 
\hline \hline
$S_{1} = \{D(\zeta), D(\zeta^{2}),  D(\zeta^{5}), D(\zeta^{6}), D(\zeta^{10})\}$ \\ 
\hline
$S_{2} = \{D(\zeta), D(\zeta^{2}),  D(\zeta^{5}), D(\zeta^{6}), D(\zeta^{9})\}$ \\ 
\hline
$S_{3} = \{D(\zeta), D(\zeta^{2}),  D(\zeta^{5}), D(\zeta^{6}), D(\zeta^{9}), D(\zeta^{11})\}$ \\ 
\hline
$S_{4} = \{D(\zeta), D(\zeta^{2}),  D(\zeta^{5}), D(\zeta^{6}), D(\zeta^{7}), D(\zeta^{11})\}$ \\ 
\hline
$S_{5} = \{D(\zeta), D(\zeta^{2}), D(\zeta^{4}), D(\zeta^{5}), D(\zeta^{6}), D(\zeta^{7}), D(\zeta^{11})\}$ \\ 
\hline
$S_{6} = \{D(\zeta), D(\zeta^{2}), D(\zeta^{4}), D(\zeta^{5}), D(\zeta^{6})\}$ \\ 
\hline
$S_{7} = \{D(\zeta), D(\zeta^{2}),  D(\zeta^{5}), D(\zeta^{6})\}$ \\ 
\hline
$S_{8} = \{D(\zeta), D(\zeta^{2}),  D(\zeta^{4}), D(\zeta^{5})\}$ \\ 
\hline
$S_{9} = \{D(\zeta^{2}), D(\zeta^{4}),  D(\zeta^{5}), D(\zeta^{6})\}$ \\ 
\hline
$S_{10} = \{D(\zeta), D(\zeta^{4}),  D(\zeta^{5})\}$ \\ 
\hline
$S_{11} = \{D(\zeta^{3}), D(\zeta^{5}), D(\zeta^{6}), D(\zeta^{7}), D(\zeta^{8}), D(\zeta^{9}), D(\zeta^{10})\}$ \\ 
\hline
$S_{12} = \{D(\zeta^{3}), D(\zeta^{4}), D(\zeta^{5}), D(\zeta^{6}), D(\zeta^{7}), D(\zeta^{8}), D(\zeta^{9}), D(\zeta^{10})\}$ \\ 
\hline
\end{longtable}
\end{center}

A $5 \times 12$ generator matrix for the binary $[12,5,4]$ LCD optimal Hom-IDD code obtained using the basis $S_{6}$ is $$\begin{pmatrix}
1 & 1 & 1 & 1 & 0 & 1 & 0 & 0 & 1 & 0 & 0 & 0 \\
0 & 1 & 0 & 0 & 0 & 1 & 1 & 1 & 0 & 1 & 1 & 0 \\
0 & 0 & 1 & 0 & 1 & 0 & 1 & 0 & 0 & 1 & 0 & 1  \\
0 & 1 & 1 & 1 & 1 & 0 & 1 & 0 & 0 & 0 & 1 & 0 \\
0 & 1 & 1 & 0 & 0 & 1 & 1 & 0 & 1 & 1 & 0 & 1
\end{pmatrix}.$$ \vspace{16pt}

\textbf{(ii)} Also, we have obtained ternary Hom-IDD codes of length $n=12$ over the finite field $\mathbb{Z}_{3}$ with various parameters, using the construction discussed above.

\begin{center}
\begin{longtable}{|c|c|c|c|}
\hline 
\textbf{Basis $S_{i}$} & \thead{\textbf{Code Description:}\\ $[n,k,d]$} & \textbf{Code Properties} & \thead{\textbf{Dual Code} \\ \textbf{Description:} \\ $[n,k,d]$} \\ 
\hline \hline
$S_{1}$ & $[12,7,3]$ & LCD & $[12,5,4]$ \\
\hline
$S_{2}$ & $[12,6,4]$ & LCD & $[12,6,4]$ \\
\hline
$S_{3}$ & $[12,5,4]$ & non-LCD & $[12,7,3]$  \\
\hline
$S_{4}$ & $[12,8,3]$ & \thead{LCD \& optimal} & $[12,4,5]$ \\
\hline
$S_{5}$ & $[12,8,3]$ & \thead{non-LCD \& optimal} & $[12,4,4]$ \\
\hline
$S_{6}$ & $[12,6,3]$ & LCD & $[12,6,3]$ \\
\hline
$S_{7}$ & $[12,6,4]$ & non-LCD & $[12,6,4]$ \\
\hline
$S_{8}$ & $[12,7,3]$ & non-LCD & $[12,5,4]$ \\
\hline
$S_{9}$ & $[12,4,4]$ & LCD & $[12,8,2]$ \\
\hline
$S_{10}$ & $[12,3,6]$ & LCD & $[12,9,1]$ \\
\hline
$S_{11}$ & $[12,2,6]$ & LCD & $[12,10,1]$ \\
\hline
$S_{12}$ & $[12,4,6]$ & \thead{LCD \& optimal} & \thead{$[12,8,3]$ \\ optimal} \\
\hline
$S_{13}$ & $[12,5,5]$ & non-LCD & $[12,7,3]$  \\
\hline
$S_{14}$ & $[12,9,2]$ & non-LCD & $[12,3,6]$ \\
\hline
\end{longtable}
\end{center}

\begin{center}
\begin{longtable}{|c|}
\hline 
\textbf{Basis $S_{i}$} \\ 
\hline \hline
$S_{1} = \{D(\zeta), D(\zeta^{2}), D(\zeta^{4}), D(\zeta^{5}), D(\zeta^{7}), D(\zeta^{9}), D(\zeta^{11})\}$ \\ 
\hline
$S_{2} = \{D(\zeta), D(\zeta^{2}), D(\zeta^{4}), D(\zeta^{5}), D(\zeta^{7}), D(\zeta^{9})\}$ \\ 
\hline
$S_{3} = \{D(\zeta), D(\zeta^{2}), D(\zeta^{4}), D(\zeta^{5}), D(\zeta^{7})\}$ \\ 
\hline
$S_{4} = \{D(\zeta), D(\zeta^{2}), D(\zeta^{3}), D(\zeta^{4}), D(\zeta^{5}), D(\zeta^{7}), D(\zeta^{9}), D(\zeta^{11})\}$ \\ 
\hline
$S_{5} = \{D(\zeta), D(\zeta^{2}),  D(\zeta^{4}), D(\zeta^{5}), D(\zeta^{7}), D(\zeta^{8}), D(\zeta^{9}), D(\zeta^{11})\}$ \\ 
\hline
$S_{6} = \{D(\zeta^{4}), D(\zeta^{5}), D(\zeta^{6}), D(\zeta^{7}), D(\zeta^{9}), D(\zeta^{11})\}$ \\ 
\hline
$S_{7} = \{D(\zeta^{2}), D(\zeta^{4}), D(\zeta^{5}), D(\zeta^{7}), D(\zeta^{9}), D(\zeta^{10})\}$ \\ 
\hline
$S_{8} = \{D(\zeta^{2}), D(\zeta^{4}), D(\zeta^{5}), D(\zeta^{7}), D(\zeta^{8}), D(\zeta^{9}), D(\zeta^{10})\}$ \\ 
\hline
$S_{9} = \{D(\zeta), D(\zeta^{4}),  D(\zeta^{7}), D(\zeta^{11})\}$ \\ 
\hline
$S_{10} = \{D(\zeta), D(\zeta^{4}),  D(\zeta^{7})\}$ \\ 
\hline
$S_{11} = \{D(\zeta), D(\zeta^{11})\}$ \\ 
\hline
$S_{12} = \{D(\zeta), D(\zeta^{2}),  D(\zeta^{4}), D(\zeta^{7})\}$ \\
\hline
$S_{13} = \{D(\zeta), D(\zeta^{2}),  D(\zeta^{4}), D(\zeta^{7}), D(\zeta^{8})\}$ \\
\hline
$S_{14} = \{D(\zeta), D(\zeta^{2}), D(\zeta^{3}), D(\zeta^{4}), D(\zeta^{5}), D(\zeta^{7}), D(\zeta^{9}), D(\zeta^{10}), D(\zeta^{11})\}$ \\ 
\hline
\end{longtable}
\end{center}

An $8 \times 12$ generator matrix for the binary $[12,8,3]$ LCD optimal Hom-IDD code obtained using the basis $S_{4}$ is $$\begin{pmatrix}
1 & 1 & 1 & 1 & 0 & 1 & 0 & 0 & 1 & 0 & 0 & 0 \\
0 & 1 & 2 & 2 & 2 & 1 & 1 & 1 & 0 & 1 & 1 & 0 \\
2 & 2 & 0 & 1 & 2 & 2 & 1 & 1 & 0 & 0 & 0 & 0 \\
0 & 0 & 2 & 0 & 1 & 2 & 0 & 2 & 2 & 1 & 0 & 1 \\
0 & 2 & 2 & 2 & 2 & 0 & 1 & 2 & 2 & 2 & 1 & 0 \\
2 & 2 & 1 & 0 & 0 & 2 & 0 & 1 & 1 & 2 & 2 & 0 \\
1 & 1 & 1 & 2 & 1 & 0 & 2 & 1 & 2 & 2 & 0 & 1 \\
0 & 0 & 1 & 0 & 0 & 1 & 1 & 1 & 0 & 2 & 2 & 2
\end{pmatrix}.$$
\end{example} \vspace{16pt}

\begin{example}
Let $K = \mathbb{Q}(\zeta)$, $\zeta$ primitive $p^{\text{th}}$-root of unity, where $p=17$. Let $\sigma(\zeta) = \zeta$ and $\tau(\zeta) = \zeta^{3}$. Then by Theorem \ref{theorem 3.8}, the map $D:O_{K} \rightarrow O_{K}$ given by $$D(\zeta) = 1 + \zeta + \zeta^{2} + \zeta^{3} + \zeta^{5} + \zeta^{7} + \zeta^{8} + \zeta^{11}$$ is a $(\sigma, \tau)$-derivation of $O_{K} = \mathbb{Z}[\zeta]$. Below, we have obtained binary Hom-IDD codes of length $n=16$ over the finite field $\mathbb{Z}_{2}$ with various parameters, using the construction discussed above.
\begin{center}
\begin{longtable}{|c|c|c|c|}
\hline 
\textbf{Basis $S_{i}$} & \thead{\textbf{Code Description:}\\ $[n,k,d]$} & \textbf{Code Properties} & \thead{\textbf{Dual Code} \\ \textbf{Description:} \\ $[n,k,d]$} \\ 
\hline \hline
$S_{1}$ & $[16,8,4]$ & LCD & $[16,8,4]$ \\
\hline
$S_{2}$ & $[16,8,4]$ & non-LCD & $[16,8,4]$ \\
\hline
$S_{3}$ & $[16,7,4]$ & LCD & $[16,9,3]$ \\
\hline
$S_{4}$ & $[16,9,3]$ & non-LCD & $[16,7,4]$ \\
\hline
$S_{5}$ & $[16,7,4]$ & non-LCD & $[16,9,3]$ \\
\hline
$S_{6}$ & $[16,6,4]$ & non-LCD & $[16,10,3]$ \\
\hline
$S_{7}$ & $[16,5,5]$ & non-LCD & $[16,11,2]$ \\
\hline
$S_{8}$ & $[16,4,7]$ & non-LCD & $[16,12,2]$ \\
\hline
$S_{9}$ & $[16,5,5]$ & LCD & $[16,11,2]$ \\
\hline
$S_{10}$ & $[16,8,3]$ & non-LCD & $[16,8,4]$ \\
\hline
$S_{11}$ & $[16,9,3]$ & LCD & $[16,7,4]$ \\
\hline
$S_{12}$ & $[16,8,4]$ & LCD & $[16,8,3]$ \\
\hline
$S_{13}$ & $[16,6,5]$ & LCD & $[16,10,3]$ \\
\hline
\end{longtable}
\end{center}

\begin{center}
\begin{longtable}{|c|}
\hline 
\textbf{Basis $S_{i}$} \\ 
\hline \hline
$S_{1} = \{D(\zeta), D(\zeta^{2}), D(\zeta^{4}), D(\zeta^{5}), D(\zeta^{6}), D(\zeta^{7}), D(\zeta^{10}), D(\zeta^{13})\}$ \\ 
\hline
$S_{2} = \{D(\zeta), D(\zeta^{2}), D(\zeta^{4}), D(\zeta^{5}), D(\zeta^{6}), D(\zeta^{7}), D(\zeta^{9}), D(\zeta^{12})\}$ \\ 
\hline
$S_{3} = \{D(\zeta), D(\zeta^{2}), D(\zeta^{5}), D(\zeta^{6}), D(\zeta^{7}), D(\zeta^{9}), D(\zeta^{12})\}$ \\ 
\hline
$S_{4} = \{D(\zeta), D(\zeta^{2}), D(\zeta^{5}), D(\zeta^{6}), D(\zeta^{7}), D(\zeta^{9}), D(\zeta^{10}), D(\zeta^{12}), D(\zeta^{14})\}$ \\ 
\hline
$S_{5} = \{D(\zeta), D(\zeta^{2}), D(\zeta^{6}), D(\zeta^{7}), D(\zeta^{9}), D(\zeta^{12}), D(\zeta^{14})\}$ \\ 
\hline
$S_{6} = \{D(\zeta), D(\zeta^{2}), D(\zeta^{6}), D(\zeta^{7}), D(\zeta^{9}), D(\zeta^{12})\}$ \\ 
\hline
$S_{7} = \{D(\zeta), D(\zeta^{2}), D(\zeta^{5}), D(\zeta^{6}), D(\zeta^{7})\}$ \\ 
\hline
$S_{8} = \{D(\zeta), D(\zeta^{2}), D(\zeta^{6}), D(\zeta^{7})\}$ \\ 
\hline
$S_{9} = \{D(\zeta), D(\zeta^{2}), D(\zeta^{6}), D(\zeta^{7}), D(\zeta^{15})\}$ \\ 
\hline
$S_{10} = \{D(\zeta), D(\zeta^{2}), D(\zeta^{6}), D(\zeta^{7}), D(\zeta^{10}), D(\zeta^{13}), D(\zeta^{14}), D(\zeta^{15})\}$ \\ 
\hline
$S_{11} = \{D(\zeta), D(\zeta^{2}), D(\zeta^{4}), D(\zeta^{5}), D(\zeta^{6}), D(\zeta^{10}), D(\zeta^{13}), D(\zeta^{14}), D(\zeta^{15})\}$ \\ 
\hline
$S_{12} = \{D(\zeta), D(\zeta^{4}), D(\zeta^{5}), D(\zeta^{6}), D(\zeta^{10}), D(\zeta^{13}), D(\zeta^{14}), D(\zeta^{15})\}$ \\ 
\hline
$S_{13} = \{D(\zeta), D(\zeta^{4}), D(\zeta^{5}), D(\zeta^{6}), D(\zeta^{10}), D(\zeta^{13})\}$ \\ 
\hline
\end{longtable}
\end{center}
\end{example}

\section{Conclusion}\label{section 5}
In this article, we studied the $(\sigma, \tau)$-derivations of number rings with their applications to coding theory. First, in Section 2, we obtained some useful results for the forthcoming sections. In Section 3, we studied $(\sigma, \tau)$-derivations of the rings of algebraic integers of quadratic and cyclotomic number fields, and that of the bi-quadratic number ring. In the quadratic case, we reformulated and reinvented the results formed in \cite{Chaudhuri}. We obtained a characterization in Subsection 3.2 for a $\mathbb{Z}$-linear map vanishing on unity to be a $(\sigma, \tau)$-derivation of the ring of algebraic integers of a $p^{\text{th}}$-cyclotomic number field. We also conjectured a criterion under which a $(\sigma, \tau)$-derivation of the ring of algebraic integers of a $p^{\text{th}}$ cyclotomic number field is inner. In Subsection 3.3, we explicitly obtained all $(\sigma, \tau)$-derivations of a bi-quadratic number ring. We also obtained a criterion under which a $(\sigma, \tau)$-derivation of a bi-quadratic number ring is inner. Thus, we have solved the twisted derivation problem in the ring of algebraic integers of a quadratic number field and in bi-quadratic number rings. Further, we conjectured a solution of the twisted derivation problem in $p^{th}$ ($p$ prime) cyclotomic number fields. Finally, in Section 4, we discussed the applications of our work in coding theory by constructing some good parameter Hom-IDD codes over finite fields.
\vspace{20pt}

\noindent \textbf{Acknowledgements}\vspace{8pt}

\noindent The second author was the ConsenSys Blockchain chair professor when the work was being carried out. He thanks ConsenSys AG for that privilege.\vspace{10pt}

\noindent \textbf{Declaration of Interest Statement}\vspace{8pt}

\noindent The authors report there are no competing interests to declare.

\bibliographystyle{plain}
\bibliography{references}
\end{document}